\documentclass{amsart}

\usepackage[mathcal]{euscript}
\usepackage{amssymb, amsfonts, xypic, amsmath}
\usepackage{amscd}
\usepackage[all]{xy}

\newtheorem{theorem}{Theorem}[section]
\newtheorem{lemma}[theorem]{Lemma}
\newtheorem{corollary}[theorem]{Corollary}

\newtheorem{proposition}[theorem]{Proposition}

\theoremstyle{definition}
\newtheorem{definition}[theorem]{Definition}

\newtheorem{remark}[theorem]{Remark}

\numberwithin{equation}{section} \numberwithin{figure}{section}

\numberwithin{equation}{section}

\newtheorem*{ack}{Acknowledgments}

\usepackage{graphicx}
\usepackage{epstopdf}

\author{Manuel Rivera and Mahmoud Zeinalian}
\newcommand{\Addresses}{{
  \bigskip
  \footnotesize

   \textsc{Manuel Rivera, Department of Mathematics, University of Miami, 1365 Memorial Drive, Coral
Gables, FL 33146}\par\nopagebreak
  \textit{E-mail address} \texttt{manuelr@math.miami.edu}

  \medskip
  \medskip

  \textsc{Mahmoud Zeinalian, Department of Mathematics, City University of New York, Lehman College, 250 Bedford Park Blvd W, Bronx, NY 10468
   }\par\nopagebreak
  \textit{E-mail address} \texttt{mahmoud.zeinalian@lehman.cuny.edu}

}}

\begin{document}

\begin{abstract}

We prove the following generalization of a classical result of Adams: for any pointed path connected topological space $(X,b)$, that is not necessarily simply connected, the cobar construction of the differential graded (dg) coalgebra of normalized singular chains in $X$ with vertices at $b$ is weakly equivalent as a differential graded associative algebra (dga) to the singular chains on the Moore based loop space of $X$ at $b$. We deduce this statement from several more general categorical results of independent interest. We construct a functor  $\mathfrak{C}_{\square_c}$ from simplicial sets to categories enriched over cubical sets with connections which, after triangulation of their mapping spaces, coincides with Lurie's rigidification functor $\mathfrak{C}$ from simplicial sets to simplicial categories. Taking normalized chains of the mapping spaces of $\mathfrak{C}_{\square_c}$ yields a functor $\Lambda$ from simplicial sets to dg categories which is the left adjoint to the dg nerve functor. For any simplicial set $S$ with $S_0=\{x\}$,  $\Lambda(S)(x,x)$ is a dga isomorphic to $\Omega Q_{\Delta}(S)$, the cobar construction on the dg coalgebra $Q_{\Delta}(S)$ of normalized chains on $S$. We use these facts to show that $Q_{\Delta}$ sends categorical equivalences between simplicial sets to maps of connected dg coalgebras which induce quasi-isomorphisms of dga's under the cobar functor, which is strictly stronger than saying the resulting dg coalgebras are quasi-isomorphic.

\end{abstract}

\title[Cubical rigidification, cobar construction, based loop space]{Cubical rigidification, the cobar construction, and the based loop space}
\maketitle

\section{Introduction}

In order to compare two different models for $\infty$-categories, Lurie constructs in \cite{Lur09} a \textit{rigidification}, or \textit{categorification}, functor $\mathfrak{C}: Set_{\Delta} \to Cat_{\Delta}$, where $Set_{\Delta}$ denotes the category of simplicial sets and $Cat_{\Delta}$ the category of simplicial categories (categories enriched over simplicial sets). For a standard $n$-simplex $\Delta^n$ the simplicial category $\mathfrak{C}(\Delta^n)$ has the set $[n]=\{0,1,...,n\}$ as objects and for any $i,j \in [n]$ with $i \leq j$ the mapping space $\mathfrak{C}(\Delta^n)(i,j)$ is isomorphic to the simplicial cube $(\Delta^1)^{\times j-i-1}$ if $i<j$, $\Delta^0$ if $i=j$, and empty if $i>j$. In particular, $\mathfrak{C}(\Delta^n)(0,n) \cong (\Delta^1)^{\times n-1}$ for $n>0$ and we think of this simplicial $(n-1)$-cube as parametrizing a family of paths in $\Delta^n$ from $0$ to $n$. Adams described in \cite{Ada52} an algebraic construction, known as the cobar construction, that when applied to a suitable differential graded coassociative coalgebra model of a simply connected space $X$ produces a differential graded associative algebra (dga) model for the based loop space of $X$. Adams' construction is based on certain geometric maps $\theta_n: I^{n-1} \to P_{0,n}|\Delta^n|$, where $P_{0,n}|\Delta^n|$ is the space of paths in the topological $n$-simplex $|\Delta^n|$ from vertex $0$ to vertex $n$, satisfying a compatibility equation that relates the cubical boundary to the simplicial face maps and the Alexander-Whitney coproduct. The definition of $\mathfrak{C}(\Delta^n)(0,n)$ resembles the construction of Adams' maps $\theta_n$ and it suggests that behind Adams' constructions there is a space level story.
\\

In this article we describe explicitly the relationship between Lurie's functor $\mathfrak{C}$ and Adams' cobar construction. As a consequence we obtain a generalization of the main theorem of \cite{Ada52} to path connected spaces with possibly non-trivial fundamental group. To achieve this, we factor the functor $\mathfrak{C}$ through a functor $\mathfrak{C}_{\square_c}: Set_{\Delta} \to Cat_{\square_c}$ from the category of simplicial sets to the category of categories enriched over cubical sets with connections. If we apply the functor of normalized cubical chains (over a fixed commutative ring $k$) to the mapping spaces of $\mathfrak{C}_{\square_c}$  we obtain a functor $\Lambda: Set_{\Delta} \to dgCat_{k}$ from simplicial sets to dg categories satisfying the following properties. The functor $\Lambda$ is the left adjoint of the dg nerve functor described by Lurie in \cite{Lur11}. Moreover, if $S$ is a $0$-reduced simplicial set, i.e. $S_0=\{x \}$, then $\Lambda(S)(x,x)$ is a dga isomorphic to $\Omega Q_{\Delta}(S)$, the cobar construction on the dg coalgebra $Q_{\Delta}(S)$ of normalized simplicial chains with Alexander-Whitney coproduct.
\\

From the properties of $\mathfrak{C}_{\square_c}$ described in the above paragraph we deduce that $\Lambda(S)(x,x)$ and $Q_{\Delta}(\mathfrak{C}(S)(x,x))$ are weakly equivalent as dga's, where $Q_{\Delta}(\mathfrak{C}(S)(x,x))$ is considered as a dga obtained by taking normalized simplicial chains on the simplicial monoid $\mathfrak{C}(S)(x,x)$. In fact, $Q_{\Delta}(\mathfrak{C}(S)(x,x))$ is a dg bialgebra (with Alexander-Whitney coproduct) but we are not concerned with the dg coalgebra structure in this article. From these results, it follows that if $f: S \to S'$ is a map between $0$-reduced simplicial sets such that $\mathfrak{C}(f): \mathfrak{C}(S) \to \mathfrak{C}(S')$ is a weak equivalence of simplicial categories (these maps are called \textit{categorical equivalences}) then $Q_{\Delta}(f): Q_{\Delta} (S) \to Q_{\Delta}(S')$ is a map of connected dg coalgebras which induces a quasi-isomorphism of dga's after applying the cobar functor. Maps $f: C \to C'$ between connected dg coalgebras which induce a quasi-isomorphism of dg algebras $\Omega f: \Omega C \to \Omega C'$ after applying the cobar functor $\Omega$ are called $\Omega$-quasi-isomorphisms. 
\\

We apply the preceding discussion to the $0$-reduced simplicial set $\text{Sing}(X,b)$ of singular simplices on a path conencted space $X$ with vertices at a fixed point $b$. From the relationships between $\mathfrak{C}$ and $\mathfrak{C}_{\square}$,  between $\mathfrak{C}_{\square}$ and the cobar functor $\Omega$, and from some basic homotopy theoretic properties of $\mathfrak{C}$, we deduce that $\Omega Q_{\Delta}(\text{Sing}(X,b))$ is weakly equivalent as a dga to the singular chains on $\Omega^M_b X$, the topological monoid of Moore loops in $X$ based at $b$. In \cite{Ada52} Adams obtained a similar statement for a simply connected space $X$ using different methods. Our statement does not assume $X$ is simply connected and therefore extends Adams' classical result. The key homotopy theoretic property of $\mathfrak{C}$ that implies our result is the following space level statement which lies at the heart of Section 2.2 of \cite{Lur09}: for any path connected pointed space $(X,b)$ there is a weak homotopy equivalence of simplicial monoids between $\mathfrak{C}(\text{Sing}(X,b) )(b,b)$ and $\text{Sing}(\Omega^M_bX)$. 
\\

We believe this extension of Adams' result has not been observed in the literature  mainly because of the historical development of the cobar construction. We highlight two situations in which the simply connected hypothesis comes into play:
\\

1) In \cite{Ada52}, Adams constructs a map of dg algebras from the cobar construction of the dg coalgebras of chains to the cubical singular chains on the based loop space. Comparison of spectral sequences was the main technique  used at the time to measure how far a chain map is from being a quasi-isomorphism. In Adams' setup, the hypotheses in Zeeman's spectral sequence comparison theorem hold if the underlying space is simply connected and fail in general for spaces with non-trivial fundamental group.
\\

2) The cobar construction is not invariant under quasi-isomorphisms of dg coalgebras. Namely, there are quasi-isomorphisms of dg coalgebras $f: C \to C'$ for which $\Omega(f): \Omega C \to \Omega C'$ is not a quasi-isomorphism of dg algebras. An explicit example is described in Proposition 2.4.3 of \cite{LoVa12}. However, the cobar construction is invariant under quasi-isomorphisms of simply connected dg coalgebras, i.e. dg coalgebras $C$ for which $C_0 \cong k $ and $C_1=0$, as shown in Proposition 2.2.7 of \cite{LoVa12}. Hence, Adams' main statement regarding the relationship between the cobar construction and the based loop space also holds if we replace singular chains on a simply connected space $X$ with the quasi-isomorphic dg coalgebra of simplicial chains associated to any simplicial set $S$ with no non-degenerate $1$-simplices whose geometric realization is weakly homotopy equivalent to $X$. This generalization of this statement to spaces with non-trivial fundamental group fails.
\\

In the non-simply connected case we go around the use of spectral sequences as described in 1) by turning the problem of showing that two dga's are quasi-isomorphic into the more fundamental problem of showing that the two simplicial monoids  $\mathfrak{C}(\text{Sing}(X,b) )(b,b)$ and $\text{Sing}(\Omega^M_bX)$  are weakly homotopy equivalent. Then by looking closely at the combinatorics we realize that the simplicial chain complex on $\mathfrak{C}(\text{Sing}(X,b) )(b,b)$ is weakly equivalent as a dga to the cobar construction $\Omega Q_{\Delta}( \text{Sing}(X,b))$.  We go around 2) by using the following observation: if $Set^0_{\Delta}$ denotes the category of simplicial sets with a single vertex, then the functor $Q_{\Delta}^K: Set^0_{\Delta} \to dgCoalg_k$ defined by $Q_{\Delta}^K( S) = Q_{\Delta}(\text{Sing}(|S|,x) )$ sends weak homotopy equivalences of simplicial sets to $\Omega$-quasi-isomorphisms of dg coalgebras. 
Notice that, in general, for any $S \in Set^0_{\Delta}$ the connected dg coalgebra of simplicial chains $Q_{\Delta}(S)$ is quasi-isomorphic to $Q_{\Delta}^K(S)$ but not $\Omega$-quasi-isomorphic. Hence, in order to preserve all the homological information of the based loop space, the chains functor should be always precomposed with a Kan replacement functor and the notion of weak equivalences of dg coalgebras should be taken to be $\Omega$-quasi-isomorphisms. \\

We now say a few words regarding how the combinatorics in the construction of $\mathfrak{C}$ is unraveled and how its cubical version $\mathfrak{C}_{\square_c}$ is constructed. For any simplicial set $S$, Dugger and Spivak computed in \cite{DS11} the mapping spaces $\mathfrak{C}(S)(x,y)$ in terms of necklaces. A necklace is a simplicial set of the form $T=\Delta^{n_1} \vee ... \vee \Delta^{n_k}$ where in the wedge the final vertex of $\Delta^{n_i}$ has been glued to the initial vertex of $\Delta^{n_{i+1}}$; a necklace in $S$  from $x$ to $y$ is a map of simplicial sets $f: T \to S$, where $T$ is a necklace, and $f$ sends the first vertex of $T$ to $x$ and the last vertex of $T$ to $y$. For any necklace $T$ one may associate functorially a simplicial cube $C(T)$ and one of the main results in \cite{DS11} is that $\mathfrak{C}(S)(x,y)$ is isomorphic to the colimit of the simplicial sets $C(T)$ over necklaces $T$ in $S$ from $x$ to $y$. It is tempting to replace the simplicial cubes $C(T)$ with standard cubical sets of the same dimension to obtain a cubical version of $\mathfrak{C}$. However, there are certain maps between necklaces that are not realized by maps of cubical sets. For example the codegeneracy map $s^1: \Delta^3 \to \Delta^2$ which collapses the edge $[1,2]$ in $\Delta^3$ yields a map between simplicial cubes $C(s^1): C(\Delta^3) \to C(\Delta^2)$ which does not correspond to a codegeneracy map between standard cubical sets. Nonetheless, $C(s^1)$ corresponds to a \textit{co-connection} morphism, whose definition is recalled in section 2. Cubical sets with connections were introduced in \cite{BH81} and can be thought of as cubical sets with extra degeneracies. In section 3 we describe explicitly the morphisms in the category of necklaces and then in section 4 we explain how cubical sets with connections arise naturally from necklaces. We use the results in sections 3 and 4 and the description of $\mathfrak{C}(S)(x,y)$ in terms of necklaces to define $\mathfrak{C}_{\square_c}$ in section 5. In section 6 we show that $\mathfrak{C}_{\square_c}$ gives rise to the functor $\Lambda$ which is the left adjoint of the dg nerve functor described by Lurie in \cite{Lur11}. Finally, in section 7 we explain how $\Lambda$ relates to the cobar construction and how to obtain an algebraic model for the based loop space of a path connected space. We also explain how our results yield a model for the free loop space of a path connected space $X$. Namely, we deduce that the coHochschild complex of the dg coalgebra $Q_{\Delta}(\text{Sing}(X,b))$ is quasi-isomorphic to the singular chains on the free loop space of $X$. 

Over a year since the results of this paper were posted on the Arxiv, two other preprints (\cite{KaVo18} and \cite{LG18}) discussing a cubical factorization of $\mathfrak{C}$ also appeared. In \cite{KaVo18}, the authors use a cubical version of $\mathfrak{C}$ to describe a cubical approach to Lurie's theory of straightening and unstraightening. In \cite{LG18}, the author discusses some homotopy theoretic properties of the category of categories enriched over cubical sets with connections using the framework of model categories. 

\begin{ack} We would like to thank Micah Miller who was a part of the early stages of this project. We would also like to thank Thomas Nikolaus, Ronnie Brown, Thomas Tradler, and Gabriel Drummond-Cole for conversations and their comments. The first author acknowledges support by the ERC via the grant StG-259118-STEIN and the excellent working conditions at \textit{Institut de Math\'ematiques de Jussieu-Paris Rive Gauche} (IMJ-PRG). The second author was partially supported by the NSF grant DMS-1309099 and would like to thank the \textit{Max Planck Institute for Mathematics} for their support and hospitality during his visits. 

\end{ack}

\section{Preliminaries}
Denote by $Set$ the category of sets. For any small category $\mathcal{C}$ denote by $Set_{\mathcal{C}}$ the category of presheaves on $\mathcal{C}$ with values in $Set$, so the objects of $Set_{\mathcal{C}}$ are functors $\mathcal{C}^{op} \to Set$ and morphisms are natural transformations between them. For example, if $\Delta$ is the category of non-empty finite ordinals with order preserving maps then $Set_{\Delta}$ is the category of \textit{simplicial sets}. We denote by $\Delta^n$ the \textit{standard n-simplex}, so $\Delta^n$ is obtained by applying the Yoneda emedding to $[n]$, namely $\Delta^n: [m] \mapsto \text{Hom}_{\Delta}( [m],[n])$. Recall that morphisms in the category $\Delta$ are generated by functions of two types: cofaces $d_i: [n] \to [n+1]$, $0 \leq i \leq n+1$, and codegeneracies $s_j: [n] \to [n-1]$, $0 \leq j \leq n-1$. The Yoneda embedding yields simplicial set morphisms between standard simplices $Y(d_i): \Delta^{n} \to \Delta^{n+1}$ and $Y(s_j): \Delta^{n} \to \Delta^{n-1}$ which we call \textit{coface} and \textit{codegeneracy (simplicial) morphisms}. We say a simplicial set $S$ is $0$-reduced if the set $S_0$ is a singleton and we denote by $Set_{\Delta}^0$ be the full subcategory of the category $Set_{\Delta}$ of simplicial sets whose objects are $0$-reduced simplicial sets.
\\

For any positive integer $n$, let $\mathbf{1}^n$ be the $n$-fold cartesian product of copies of the category $\mathbf{1}=\{0,1\}$ which has two objects and one non-identity morphism. Denote by $\mathbf{1}^0$ the category with one object and one morphism. We will consider presheaves over the category $\square_c$ which is defined as follows. The objects of $\square_c$ are the categories $\mathbf{1}^n$ for $n=0,1,2,...$. The morphisms in $\square_c$ are generated by functors of the following three kinds:
\\
\textit{cubical co-face functors} $\delta^{\epsilon}_{j,n}: \mathbf{1}^n \to \mathbf{1}^{n+1}$, where $j=0,1,...,n+1$, and $\epsilon \in \{0,1\}$, defined by
\begin{eqnarray*}
\delta^{\epsilon}_{j,n}(s_1,...,s_n)=(s_1,...,s_{j-1},\epsilon,s_j,...,s_n),
\end{eqnarray*}
\textit{cubical co-degeneracy functors} $\varepsilon_{j,n}: \mathbf{1}^n \to \mathbf{1}^{n-1}$, where $j=1,...,n$, defined by
\begin{eqnarray*}
\varepsilon_{j,n}(s_1,...,s_n)=(s_1,...,s_{j-1},s_{j+1},...,s_n), \text{ and }
\end{eqnarray*}
\textit{cubical co-connection functors} $\gamma_{j,n}: \mathbf{1}^n \to \mathbf{1}^{n-1}$, where $j=1,...,n-1$, $n\geq 2$, defined by
\begin{eqnarray*}
\gamma_{j,n}(s_1,...,s_n)=(s_1,...,s_{j-1},\text{max}(s_j,s_{j+1}),s_{j+2},...,s_n).
\end{eqnarray*}
\\

Objects in the category $Set_{\square_c}$ are called \textit{cubical sets with connections} and were introduced by Brown and Higgins in \cite{BH81}. For any cubical set with connections $K$ we have a collection of sets $\{K_n:= K(\mathbf{1}^n)\}_{n \in \mathbb{Z}_{\geq 0}}$ together with \textit{cubical face maps} $\partial_{j,n}^{\epsilon}:=K(\delta^{\epsilon}_{j,n}): K_{n+1} \to K_n$, \textit{cubical degeneracy maps} $E_{j,n}:=K(\varepsilon_{j,n}): K_{n-1} \to K_n$, and \textit{connections}  $\Gamma_{j,n}:=K(\gamma_{j,n}): K_{n-1} \to K_n$. For simplicity we often drop the second index in this notation and, for example, write $\partial_j$ instead of $\partial_{j,n}$. Elements of $K_n$ are called $n$-cells. The structure maps satisfy certain identities described in \cite{BH81}. The \textit{standard n-cube with connections} $\square_c^n$ is the presheaf on $\square_c$ represented by $\mathbf{1}^n$, namely, $\text{Hom}_{\square_c}( \text{ \_ }, \mathbf{1}^n): \square_c^{op} \to Set$. 
\\

For a fixed commutative unital ring $k$ denote by $Ch_k$ the category of non-negatively graded chain complexes over $k$. The tensor product over $k$ defines on $Ch_k$ a symmetric monoidal structure. We have normalized chains functors  $Q_{\Delta}: Set_{\Delta} \to Ch_k$ and $Q_{\square_c}: Set_{\square_c} \to Ch_k$. The definition of $Q_{\Delta}$ is standard; we recall the definition of $Q_{\square_c}$ following \cite{Ant02}. First let $C_*K$ be the chain complex such that $C_nK$ is the free $k$-module generated by elements of $K_n$ with differential $\partial: K_n \to K_{n-1}$ defined on $\sigma \in K_n$ by $\partial (\sigma):= \sum_{j=1}^n(-1)^{j}(\partial^1_{j,n-1}(\sigma) - \partial^0_{j,n-1}(\sigma))$. Let $D_nK$ be the submodule of $C_nK$ which is generated by those cells in $K_n$ which are the image of a degeneracy or of a connection map $K_{n-1} \to K_n$. The graded module $D_*K$ forms a subcomplex of $C_*K$. Define $Q_{\square_c}(K)$ to be the quotient chain complex $C_*K/D_*K$. 
\\

The functor $Q_{\Delta}: Set_{\Delta} \to Ch_k$ lifts to a functor $Q_{\Delta}: Set_{\Delta} \to dgCoalg_k$, where $dgCoalg_k$ is the category of dg coalgebras over $k$, via the Alexander-Whitney construction as recalled in Section 7. There is a slight abuse of notation throughout the article: depending on the context $Q_{\Delta}(S)$ may be considered as a chain complex or as a dg coalgebra. For example, by $\Omega Q_{\Delta}(S)$ we mean the cobar construction of $Q_{\Delta}(S)$ considered as a dg coalgebra.
\\

The category $Set_{\Delta}$ has a symmetric monoidal structure given by the cartesian product of simplicial sets.  We will use the following (non-symmetric) monoidal structure on $Set_{\square_c}$: for cubical sets with connections $K$ and $K'$ define
\begin{eqnarray*}
K \otimes K' := \underset{\sigma: \square_c^n \to K, \tau: \square_c^m \to K'} {\text{colim}} \square^{n+m}_c.
\end{eqnarray*}

Using the above monoidal structures we may define $Cat_{\Delta}$ the category of small categories enriched over simplicial sets; these are called \textit{simplicial categories}. Similarly denote by $Cat_{\square_c}$ the category of small categories enriched over cubical sets with connections; these are called \textit{cubical categories with connections}. We will also consider the category $dgCat_k$ of small categories enriched over chain complexes over $k$; these are called \textit{dg categories}.
\\

The symbol $\cong$ will always denote isomorphism and $\simeq$ will mean weakly equivalent (in the derived sense) whenever there is a notion of weak equivalence in the underlying category. Namely, we write $A \simeq B$ if there is a zig-zag of weak equivalences between $A$ and $B$.

\section{The category of necklaces}
We follow \cite{DS11} for the next definitions and notation. A \textit{necklace} $T$ is a simplicial set of the form $T=\Delta^{n_1} \vee ... \vee \Delta^{n_k}$ where $n_i \geq 0$ and in the wedge the final vertex of $\Delta^{n_i}$ has been glued to the initial vertex of $\Delta^{n_{i+1}}$. Each $\Delta^{n_i}$ is called a \textit{bead} of $T$. Since the beads of $T$ are ordered and the vertices of each bead $\Delta^{n_i}$ are ordered as well, there is a canonical ordering on the set $V_T$ of vertices of any necklace $T$. We denote by $\alpha_T$ and $\omega_T$ the first and last vertices of the necklace $T$. A morphism $f: T \to T'$ of necklaces is a map of simplicial sets which preserves the first and last vertices. We say a necklace $\Delta^{n_1} \vee ... \vee \Delta^{n_k}$ is of \textit{preferred form} if  $k=0$ or each $n_i \geq 1$. Let $T= \Delta^{n_1} \vee ... \vee \Delta^{n_k}$  be a necklace in preferred form. Denote by $b_T$ the number of beads in $T$. A \textit{joint} of $T$ is either an initial or a final vertex in some bead. Given a necklace $T$ write $J_T$ for the subset of $V_T$ consisting of all the joints of $T$. For any two vertices $a,b \in V_T$ we write $V_T(a,b)$ and $J_T(a,b)$ for the set of vertices and joints between $a$ and $b$ inclusive. Note that there is a unique subnecklace $T(a,b) \subseteq T$ with joints $J_T(a,b)$ and vertices $V_T(a,b)$. Denote by $Nec$ the category whose objects are necklaces in preferred form and morphisms are morphisms of necklaces. Note that $Nec$ is a full subcategory of $Set^{*,*}_{\Delta}=\partial \Delta^1 \downarrow Set_{\Delta}$. 

\begin{proposition} Any non-identity morphism in $Nec$ is a composition of morphisms of the following type
\begin{item}
\item (i)  $f: T \to T'$ is an injective morphism of necklaces and $ |V_{T'}-J_{T'}|-|V_T-J_T| =1$
\\
\item (ii) $f: \Delta^{n_1} \vee ... \vee \Delta^{n_k} \to \Delta^{m_1} \vee ... \vee \Delta^{m_k}$ is a morphism of necklaces of the form $f=f_1 \vee ... \vee f_k$ such that for exactly one $p$, $f_p: \Delta^{n_p} \to \Delta^{m_p}$ is a codegeneracy morphism (so $m_p=n_p-1$) and for all $i \neq p$, $f_i: \Delta^{n_i}  \to \Delta^{m_i}$ is the identity map of standard simplices (so $n_i=m_i$ for $i \neq p$)
\\
\item (iii) $f: \Delta^{n_1} \vee ...\vee \Delta^{n_{p-1}} \vee \Delta^1 \vee \Delta^{n_{p+1}} \vee... \vee  \Delta^{n_k} \to \Delta^{n_1} \vee ...\vee \Delta^{n_{p-1}} \vee \Delta^{n_{p+1}} \vee... \vee  \Delta^{n_k}$ is a morphism of necklaces such that $f$ collapses the $p$-th bead $\Delta^1$  in the domain to the last vertex of the $(p-1)$-th bead in the target and the restriction of $f$ to all the other beads is injective. 
\end{item}
\end{proposition}
\begin{proof}

We prove that any non-identity morphism of necklaces $f: T\to T'$ is a composition of morphisms of type (i), (ii), and (iii) by induction on $b_T$, the number of beads of $T$. If $b_T=1$, then we must have $b_{T'}=1$ as well, so $f$ is a morphism of simplicial sets between standard simplices which preserves first and last vertices. It follows that $f$ is a composition of (simplicial) coface and codegeneracy morphisms. Cofaces and codegeneracies between standard simplices are morphisms of necklaces of type (i) and of type (ii) or (iii), respectively. Assume we have shown the proposition for $b_T \leq k$ and suppose $b_T=k+1$. Let $V_T=\{x_0,...,x_p\}$ be the vertices of $T$ and $x_i \preceq x_{i+1}$. Let $x_{j_0}$ be the last vertex of the first bead of $T$, so $T=T(x_0, x_{j_0})\vee T(x_{j_0}, x_p)$ where $T(x_0, x_{j_0})$ has one bead and $T(x_{j_0}, x_p)$ has $k$ beads. Let $T_f=T'(f(x_0),f(x_{j_0})) \vee T'(f(x_{j_0}),f(x_p))$. We have an injective morphism of necklaces $t: T_f \to T'$ (notice that it is possible for $T_f \not =T'$ since $f(x_{j_0})$ might not be a joint of $T'$). It follows that $f=t \circ (g\vee h)$ where $g: T(x_0,x_{j_0}) \to T'(f(x_0), f(x_{j_0}))$ and $h : T(x_{j_0}, x_p) \to T'(f(x_{j_0}), f(x_p))$ are the morphisms of necklaces induced by restricting $f$ to $T(x_0, x_{j_0})$ and $T(x_{j_0}, x_p)$ respectively. By the induction hypothesis each of $g$ and $h$ is a composition of morphisms of type (i), (ii), and (iii) and this implies that $g \vee h$ is a composition of such morphisms as well. In fact, we have 
\begin{eqnarray*}
g \vee h=  (id_{T'(f(x_0),f(x_{j_0}))} \vee h) \circ (g \vee id_{T(x_{j_0},x_p)})
\end{eqnarray*}
and, clearly, the wedge of an identity morphism and a morphism which is a composition of morphisms of type (i), (ii), and (iii) is again a morphism of such form. 

To conclude the proof we show that $t: T_f \to T'$ is of the desired form. More generally, let us prove that any non-identity injective morphism of necklaces $t: R \to R'$ is a composition of morphisms of type (i) by induction on the integer $l(R,R'):=|V_{R'}-J_{R'}|-|V_{R}-J_{R}|$. If $l(R,R')=1$ then $t$ is of type (i). Assume we have shown the claim for $l(R,R')=k$. Suppose $t: R \to R'$ is injective and $l(R,R')=k+1$, then we have two cases: either (a) $J_{R'}=t(J_R)$ or (b) $J_{R'} \subset t(J_{R})$. In case (a), it follows that both $R$ and $R'$ have the same number of beads, thus  $t=  i \circ j$  for inclusions of necklaces $j: R \to S$, $i: S \to R'$ where $S$ is the subnecklace of $R'$ spanned by $t(V_R) \cup \{v\}$ and $v$ is the smallest element of $V_{R'}-t(V_R)$. Then $j$ is of type (i) and $i$ is a composition of morphisms of type (i) by the induction hypothesis.  For case (b), let $ t(J_R) - J_{R'}=\{t(x_{i_1}), ..., t(x_{i_n})\}$ and consider the unique subnecklace $S$ of $R'$ defined by $V_S=t(V_R)$ and $J_S=t(J_R)-\{t(x_{i_1})\}$. Then we have $t=i \circ j$ for inclusions of necklaces $j: R \to S$, $i: S \to R'$ with $j$ of type (i) and $i$ a composition of type (i) morphisms by the induction hypothesis. 
\end{proof}

\begin{remark}
Let us consider type (i) morphisms of the form $f: T \to \Delta^p$ for some integer $p\geq 1$.  If $b_T=1$ then we have an injective map of simplicial sets $f: \Delta^{p-1} \to \Delta^{p}$ which sends the first (resp. last) vertex of $\Delta^{p-1}$ to the first (resp. last) vertex of $\Delta^{p}$. The morphism $f$ determines a $(p-1)$-simplex of the simplicial set $\Delta^p$, i.e. an element of $(\Delta^p)_{p-1}$. There are $p+1$ non-degenerate elements in $(\Delta^p)_{p-1}$, however only $p-1$ of these can correspond to $f$ based on the constraint that $f$ must preserve first and last vertices, namely, all the faces of the unique non-degenerate element in $(\Delta^p)_p$ except the first and last. If $b_T>1$ then there is a joint $v \in J_T$ such that $f(v) \not \in J_{T'}$. Moreover, since $f$ is injective and $ |V_{T'}-J_{T'}|-|V_T-J_T| =1$, we have $f (J_T-\{v\} ) =J_T'$ and $f(V_T)=V_{T'}$.  It follows that $b_T=2$ and the image of $f$ is a subnecklace $T'_1 \vee T'_2$ of $\Delta^p$ starting and ending with the first and last vertices of $\Delta^p$, respectively, and containing all the vertices of $\Delta^p$. Hence, we have $T'_1 \vee T'_2=\Delta^{p-i} \vee \Delta^{i}$ for some $0<i<p$ and each of these subnecklaces of $\Delta^p$ corresponds to a unique term in the formula for the Alexander-Whitney diagonal $Q_{\Delta}(\Delta^p) \to Q_{\Delta}(\Delta^p) \otimes Q_{\Delta}(\Delta^p)$ applied to the generator represented by the unique non-degenerate $p$-simplex in $(\Delta^p)_p$.
\end{remark}

\section{The functor $C_{\square_c}: Nec \to Set_{\square_c}$}
There is a functor $C_{\square_c}: Nec \to Set_{\square_c}$ which associates functorially to any $\Delta^{n_1} \vee... \vee \Delta^{n_k} \in Nec$ a standard cube with connections of dimension $n_1+...+n_k-k$. The goal of this section is to define this functor carefully in a way which will be useful later. We start by defining a functor $P: Nec \to Cat$ where $Cat$ is the category of small categories. Given a necklace $T$ and two vertices $a,b \in V_T$ we may define a small category $P_T(a,b)$ whose objects are subsets $X \subseteq V_T(a,b)$ such that $J_T(a,b) \subseteq X$ and morphisms are inclusions of sets. For any necklace $T \in Nec$ let $P(T) = P_T(\alpha, \omega)$ where $\alpha, \omega \in V_T$ are the first and last vertices of $T$.  Let $f: T \to T'$ be a morphism in $Nec$, so $f$ is a map of simplicial sets such that $f(\alpha)=\alpha'$ and $f(\omega)=\omega'$ where $\alpha, \omega \in V_T$ and $\alpha', \omega' \in V_{T'}$ are the first and last vertices of $T$ and $T'$, respectively. Notice that we have an inclusion $J_{T'} \subseteq f(J_T)$. Thus $f$ induces a functor $P_f: P_T(\alpha, \omega) \to P_{T'}(\alpha', \omega')$ defined on objects by $P_f(X)=f(X)$ and on morphisms by the induced inclusion of sets. This yields a functor $P: Nec \to Cat$. We might think of the objects of $P(T)$ as strings of $0$'s and $1$'s as discussed below. This interpretation will yield a functor $P_{\mathbf{1}}$ which is naturally isomorphic to $P$. We define a total order on the vertices of a necklace by setting $a \preceq b$ if there is a directed path from $a$ to $b$. 

\begin{proposition} For any necklace $T$ and any  $a,b\in V_T$ such that $a \preceq b$, there is an isomorphism of categories $\phi_T: P_T(a,b) \cong \mathbf{1}^{N}$ where $N=|V_T(a,b) - J_T(a,b)|$. 
\end{proposition}
\begin{proof}
Let $V_T(a,b)-J_T(a,b)=\{y_1,...,y_N\}$ and $y_i \preceq y_{i+1}$ for $i=1,...,{N-1}$. Given any object $X$ of $P_T(a,b)$ (so $J_T(a,b) \subseteq X \subseteq V_T(a,b)$) we define $\phi_T(X):=(\phi_T^1(X),...,\phi_T^{N}(X))$ to be the object in the category $\mathbf{1}^N$ where, for $1 \leq i \leq N$,  we have $\phi_T^i(X)=1$ if $y_i \in X$ and $\phi_T^i(X)=0$ if $y_i \not\in X$. Given a morphism $f: X \to Y$ in $P_T(a,b)$ (so $f$ is an inclusion of sets) we have an induced morphism $\phi_T(f): \phi_T(X) \to \phi_T(Y)$ defined by $\phi_T(f):= (\phi^1_T(f), ...,\phi^N_T(f))$ where, for $1 \leq i\leq N$, $\phi^i_T(f): \phi^i_T(X) \to \phi^i_T(Y)$ is the unique non-identity morphism in $\mathbf{1}$ if $\phi^i_T(X)=0$ and $\phi^i_T(Y)=1$, and $\phi^i_T(f)$ is an identity morphism otherwise.  It is clear that the functor $\phi_T: P_T(a,b) \to \mathbf{1}^N$ is an isomorphism of categories.
\end{proof}

Consider the functor $P_{\mathbf{1}}: Nec \to Cat$ defined on objects by $P_{\mathbf{1}}(T)=\mathbf{1}^{|V_T-J_T|}$ and on morphisms $f: T\to T'$ by $P_{\mathbf{1}}(f)= \phi_{T'} \circ P(f) \circ \phi_T^{-1}: \mathbf{1}^{|V_T-J_T|} \to \mathbf{1}^{|V_{T'}-J_{T'}|}$. The above proposition implies that $P_{\mathbf{1}}$ is naturally isomorphic to $P$. In the following proposition we describe explicitly the functor $P_{\mathbf{1}}(f)$ for morphisms $f: T \to T'$ of type (i), (ii), and (iii) as in Proposition 3.1.

\begin{proposition} Let $f: T \to T'$ be a morphism in $Nec$ and let $N=|V_T-J_T|$.
\begin{enumerate}
\item If $f$ is of type (i) then $P_{\mathbf{1}}(f): \mathbf{1}^N \to \mathbf{1}^{N+1}$ is a cubical co-face functor. 
\item If $f$ is of type (ii) then $P_{\mathbf{1}}(f): \mathbf{1}^N \to \mathbf{1}^{N-1}$ is either a cubical co-connection functor or a cubical co-degeneracy functor.
\item If $f$ is of type (iii) then $P_{\mathbf{1}}(f): \mathbf{1}^N \to \mathbf{1}^N$ is the identity functor. 
\end{enumerate}
\end{proposition}
\begin{proof} For any morphism of necklaces $f: T \to T'$ we have $J_{T'} \subseteq f(J_T)$. For $f: T\to T'$ of type (i) we prove below that if $J_{T'} \subset f(J_T)$  then $P_{\mathbf{1}}(T)(f)$ is a cubical co-face functor $\delta^1_{j,N}$ and if $J_{T'}=f(J_T)$ then $P_{\mathbf{1}}(T)(f)$ is a cubical co-face functor $\delta^0_{j,N}$. A morphism $f: T \to T'$ of type (ii) collapses two vertices $v$ and $w$ of $T$ into a vertex $v'$ of $T'$ and is injective on $V_T-\{v,w\}$. We prove below that if $v'\not \in J_{T'}$ then $P_{\mathbf{1}}(T)(f)$ is a cubical co-connection functor $\gamma_{j,N}$ and if $v' \in J_{T'}$ then $P_{\mathbf{1}}(T)(f)$  is a cubical co-degeneracy functor $\varepsilon_{j,N}$. The proof for the third part of the proposition will be straightforward.
\begin{enumerate}
\item Let $f : T \to T'$ be of type (i) and write $\{y'_1,...,y'_{N+1}\}= V_{T'}-J_{T'}$ where $y'_i \preceq y'_{i+1}$. We have $J_{T'} \subseteq f(J_T)$ since $f$ is a morphism of necklaces. If $J_{T'} \subset f(J_T)$ then there is $v \in J_T$ such that $f(v)=y'_j\in V_{T'}-J_{T'}$ for some $j \in \{1,...,N+1\}$ and $f(J_T-\{v\}) \subseteq J_T'$. Then for any object $X$ in $P(T)$, $v \in J_T \subseteq X$ so $y_j=f(v) \in f(X)$. Using the fact that $f$ is injective and identifying objects $X$ in $P(T)$ with sequences of $0$'s and $1$'s via the isomorphism $\phi_T: P(T) \cong \mathbf{1}^N$ we see that $P_{\mathbf{1}}(f): \mathbf{1}^N \to \mathbf{1}^{N+1}$ is given on objects by 
\begin{eqnarray*}
 P_{\mathbf{1}}(f)(s_1,...,s_N)=(s_1,....,s_{j-1},1,s_j,...,s_N)
\end{eqnarray*} 
and on morphisms $\lambda= (\lambda_1,...,\lambda_N): (s_1,...,s_N) \to (s'_1,...,s'_N)$ by
\begin{eqnarray*}
P_{\mathbf{1}}(f)(\lambda)=(\lambda_1,...,\lambda_{j-1},id_1,\lambda_j,...,\lambda_N).
\end{eqnarray*}
Thus $P_{\mathbf{1}}(f)$ is the cubical co-face functor $\delta^1_{j,N}$.
\\
\\
If $J_{T'}=f(J_T)$ then there exists exactly one $j \in \{1,...,N+1\}$ such that $f^{-1}(y'_j)=\emptyset$. Then for any object $X$ in $P(T)$, $y'_j$ will never be an element of $f(X)$. Using the fact that $f$ is injective and identifying objects $X$ in $P(T)$ with sequences of $0$'s and $1$'s via the isomorphism $\phi_T: P(T) \cong \mathbf{1}^N$  we see that $P_{\mathbf{1}}(f): \mathbf{1}^N \to \mathbf{1}^{N+1}$ is given on objects by
\begin{eqnarray*}
 P_{\mathbf{1}}(f)(s_1,...,s_N)=(s_1,....,s_{j-1},0,s_j,...,s_N)
\end{eqnarray*} 
and on morphisms $\lambda= (\lambda_1,...,\lambda_N): (s_1,...,s_N) \to (s'_1,...,s'_N)$ by
\begin{eqnarray*}
P_{\mathbf{1}}(f)(\lambda)=(\lambda_1,...,\lambda_{j-1},id_0,\lambda_j,...,\lambda_N).
\end{eqnarray*}
It follows that $P_{\mathbf{1}}(f)$ is the cubical co-face functor $\delta^0_{j,N}$.

\item Let $f : T \to T'$ be of type (ii) and write $\{y_1,...,y_N\}=V_T-J_T$ where $y_i \preceq y_{i+1}$ and $\{y'_1,...,y'_{N-1}\}= V_{T'}-J_{T'}$ where $y'_i \preceq y'_{i+1}$. There exists $v' \in V_{T'}$  such that $f^{-1}(v')=\{v,w\}$ for some $v,w \in V_T$ and $|f^{-1}(x')|=1$ for all $x' \in V_{T'}-\{v'\}$. Note that $v$ and $w$ are consecutive vertices in the $p$-th bead of $T$. We have two cases: either $v' \in V_{T'}-J_{T'}$ or $v' \in J_{T'}$.
\\
\\
If $v' \in V_{T'}-J_{T'}$, then $v,w \in V_T-J_T$ so we may write $v=y_j$ and $w=y_{j+1}$ for some $j \in \{1,...,N-1\}$. Hence, for any object $X$ of $P(T)$ we have that if $X \cap \{y_j,y_{j+1}\} \neq \emptyset$ then $v' \in f(X)$ and if $X \cap \{y_j,y_{j+1}\} = \emptyset$ then $v' \not \in f(X)$. By identifying objects $X$ in $P(T)$ with sequences of $0$'s and $1$'s via the isomorphism $\phi_T: P(T) \cong \mathbf{1}^N$ we see that $P_{\mathbf{1}}(f): \mathbf{1}^N \to \mathbf{1}^{N-1}$ is given on objects by
\begin{eqnarray*}
P_{\mathbf{1}}(f)(s_1,...,s_N)=(s_1,....s_{j-1},max(s_j,s_{j+1}),s_{j+2},...,s_N)
\end{eqnarray*} 
and on morphisms $\lambda= (\lambda_1,...,\lambda_N): (s_1,...,s_N) \to (s'_1,...,s'_N)$ by
\begin{eqnarray*}
P_{\mathbf{1}}(f)(\lambda)=(\lambda_1,...,\lambda_{j-1},\sigma_{j,j+1}, \lambda_{j+2},...,\lambda_N),
\end{eqnarray*}
where $\sigma_{j,j+1}$ is the unique morphism $max(s_j,s_{j+1})\to max(s'_j,s'_{j+1})$ in the category $\mathbf{1}$. It follows that $P_{\mathbf{1}}(f)$ is the cubical co-connection functor $\gamma_{j,N}$. 
\\
\\
If $v' \in J_{T'}$, we may assume without loss of generality that $w\in J_T$ and $v=y_j \in V_T-J_T$  for some $j \in \{1,...,N\}$. Let $X$ be any object of $P(T)$. Every element of $X-\{y_j\}$ corresponds to a unique element in $f(X)$ via $P(f)$ (since $f$ is of type (ii)) and if $y_j \in X$ then $P(f)$ sends $y_j$ to the joint $v' \in f(X)$. By identifying objects $X$ in $P(T)$ with sequences of $0$'s and $1$'s via the isomorphism $\phi: P(T) \cong \mathbf{1}^N$ we see that $P_{\mathbf{1}}(f): \mathbf{1}^N \to \mathbf{1}^{N-1}$ is given on objects by
\begin{eqnarray*}
P_{\mathbf{1}}(f)(s_1,...,s_N) = (s_1,...,s_{j-1},s_{j+1},...,s_N)
\end{eqnarray*}
and on morphisms $\lambda= (\lambda_1,...,\lambda_N): (s_1,...,s_N) \to (s'_1,...,s'_N)$ by
\begin{eqnarray*}
P_{\mathbf{1}}(f)(\lambda)=(\lambda_1,...,\lambda_{i-1},\lambda_{i+1},...,\lambda_N).
\end{eqnarray*}
It follows that $P_{\mathbf{1}}(f)$ is the cubical co-degeneracy functor $\varepsilon_{j,N}$. 
\item If $f$ is of type (iii) then $|V_T|=|V_{T'}|+1$ and the injectivity of $f$ only fails when it collapses two joints (the endpoints of the $p$-th bead $\Delta^1$) to a joint in $T'$. Under the isomorphism $\phi_T: P(T) \cong \mathbf{1}^N$ this collapse does not have any effect since given an object  $X$ of $P(T)$ the entries in the string $\phi_T(X)$ of $0$'s and $1$'s only indicate which non-joint vertices of $T$ are in $X$. It follows that $P_{\mathbf{1}}(f): \mathbf{1}^N \to \mathbf{1}^N$ is the identity functor. 
\end{enumerate}
\end{proof}
\begin{remark}
Consider two morphisms of necklaces $f: U \to T$ and $g:V\to T$. If $f$ and $g$ are both of type (i) and $f \neq g$ then $P_{\mathbf{1}}(f) \neq P_{\mathbf{1}}(g)$. If $f$ and $g$ are of both of type (ii) and $f \neq g$ we may have $P_{\mathbf{1}}(f)=P_{\mathbf{1}}(g)$. For example, let $U=W \vee \Delta^{m+1} \vee \Delta^n \vee W'$, $V=W \vee \Delta^m \vee \Delta^{n+1}\vee W'$, $T=W\vee \Delta^m \vee \Delta^n \vee W'$, for any two necklaces $W$ and $W'$. Consider the maps $f=id_W \vee s_{m+1} \vee id_{\Delta^n} \vee id_{W'} $ and $g= id_W \vee id_{\Delta^m} \vee s_1 \vee id_{W'}$, where $s_{m+1}: \Delta^{m+1} \to \Delta^{m}$ and $s_1: \Delta^{n+1} \to \Delta^n$ are the last and first (simplicial) codegeneracy morphisms respectively. It follows that $P_{\mathbf{1}}(f)=P_{\mathbf{1}}(g)$. The identification of these two morphisms after applying $P_{\mathbf{1}}$ should be compared with the identification in the colimit defining the monoidal structure of the category of cubical sets with connections discussed in the next section. Finally, if $f$ and $g$ are of type (iii), then we always have $P_{\mathbf{1}}(f)=P_{\mathbf{1}}(g)$. 
\end{remark}

\begin{corollary} The functor $P_{\mathbf{1}}: Nec \to Cat$ factors as a composition $Nec \to \square_c \hookrightarrow Cat$.
\end{corollary}
\begin{proof} For any object $T$ in $Nec$, $P_{\mathbf{1}}(T)= \mathbf{1}^N$ is an object of $\square_c$ and, by Proposition 4.2, for any morphism $f$ in $Nec$, $P_{\mathbf{1}}(f)$ is a morphism in $\square_c$.
\end{proof}
Hence, we may consider $P_{\mathbf{1}}$ as a functor from $Nec$ to $\square_c$. Finally, we define a functor from the category of necklaces to the category of cubical sets as follows.
\begin{definition}
Define the functor $C_{\square_c}: Nec \to  Set_{\square_c}$ to be the composition of functors $C_{\square_c}:=Y \circ P_{\mathbf{1}}$ where $Y: \square_c \to \text{Hom}_{Cat}((\square_c)^{op},Set)=Set_{\square_c}$ is the Yoneda embedding.
\end{definition}
Note that for any $T$ in $Nec$, $C_{\square_c}(T)$ is the standard cube with connections $\square^N_c$ where $N=|V_T-J_T|$. 
\begin{remark}
All non-degenerate cells of $C_{\square_c}(T)$ can be realized by injective maps of necklaces $T' \to T$. More precisely, for every non-degenerate cell $\sigma \in C_{\square_c}(T)_n$ there is a necklace $T_{\sigma}$, with $|V_{T_{\sigma}}- J_{T_{\sigma}}|=n$ together with an injective map of necklaces $\iota_{\sigma}: T_{\sigma} \to T$ such that the induced map of cubical sets with connections
\begin{eqnarray*}
\square^n_c \cong C_{\square_c}(T_{\sigma}) \xrightarrow{C_{\square_c}(\iota_{\sigma})} C_{\square_c}(T)
\end{eqnarray*}
corresponds to the cell $\sigma$. Notice $T_{\sigma}$ is not unique, since any other $T'_{\sigma}$ for which there is a map $T'_{\sigma} \to T_{\sigma}$ of type (iii) also works.
\end{remark}

\section{The cubical rigidification functor $\mathfrak{C}_{\square_c}: Set_{\Delta} \to Cat_{\square_c}$}
The goal of this section is to show that the functor $\mathfrak{C}: Set_{\Delta} \to Cat_{\Delta}$ defined by Lurie factors naturally through categories enriched over cubical sets with connections via a functor  $\mathfrak{C}_{\square_c}: Set_{\Delta} \to Cat_{\square_c}$. More precisely, we construct functors $\mathfrak{C}_{\square_c}: Set_{\Delta} \to Cat_{\square_c}$ and $\mathfrak{T}: Cat_{\square_c} \to Cat_{\Delta}$ such that $\mathfrak{T} \circ \mathfrak{C}_{\square_c}$ is naturally isomorphic to $\mathfrak{C}$.
\begin{definition}
For any simplicial set $S$ we define a category $\mathfrak{C}_{\square_c}(S)$ enriched over cubical sets with connections.  Define the objects of  $\mathfrak{C}_{\square_c}(S)$ to be the vertices of $S$, i.e. the elements of $S_0$.  For any $x,y \in S_0$ define
\begin{eqnarray*}
 \mathfrak{C}_{\square_c}(S)(x,y):= \underset{T \to S \in (Nec \downarrow S)_{x,y}}{\text{colim}} C_{\square_c}(T)
 \end{eqnarray*}
 where $(Nec \downarrow S)_{x,y}$ is the category whose objects are morphisms $f: T \to S$ for some $T \in Nec$ such that $f(\alpha_T)=x$ and $f(\omega_T)=y$. For any $x,y,z \in S_0$ the composition law
\begin{eqnarray*}
\mathfrak{C}_{\square_c}(S)(y,z)  \otimes \mathfrak{C}_{\square_c}(S)(x,y) \to \mathfrak{C}_{\square_c}(S)(x,z)
\end{eqnarray*}
is induced as follows. Note that given $T \to S \in (Nec \downarrow S)_{x,y}$  and $U \to S \in (Nec \downarrow S)_{y,z}$, we obtain $T \vee U \to S \in  (Nec \downarrow S)_{x,z}$. Then the composition
\begin{eqnarray*}
C_{\square_c}(U) \otimes C_{\square_c}(T) \to C_{\square_c}((T \vee U)(\alpha_U, \omega_U)) \otimes C_{\square_c}((T \vee U)(\alpha_T,\omega_T)) \to C_{\square_c}(T \vee U)
\end{eqnarray*}
of morphisms of cubical sets with connections induces the desired composition law after taking colimits. Recall that $(T \vee U)(\alpha_U, \omega_U)$ denotes the unique subnecklace of $T \vee U$ with joints $J_{T \vee U}(\alpha_U,\omega_U)$ and vertices $V_{T \vee U}(\alpha_U, \omega_U)$. It follows from Remark 4.3 that the above composition passes to the colimit and yields a well defined composition rule. Finally, it is clear that $\mathfrak{C}_{\square_c}(S)$ is functorial in $S$. 
 \end{definition}
 
 \begin{remark}
 
The set of $n$-cells in $\mathfrak{C}_{\square_c}(S)(x,y)$ is 
\begin{eqnarray*}
\Big(\bigsqcup_{(T\to S) \in (Nec \downarrow S)_{x,y}}C_{\square_c}(T)_n \Big)/ \sim
\end{eqnarray*}
where the equivalence relation is generated by $(t: T \to S, \sigma) \sim (t': T'\to S, \sigma')$ if there is a map of necklaces $f: T \to T'$ such that $t=t' \circ f$ and $C_{\square_c}(f)(\sigma)=\sigma'$. Here $t: T \to S$ and $t': T' \to S$  are objects in $(Nec \downarrow S)_{x,y}$, and $\sigma$ and $\sigma'$ are $n$-cells in  $C_{\square_c}(T)$ and $C_{\square_c}(T')$, respectively. Any \textit{non-degenerate} $n$-cell $[t: T \to S, \sigma] \in \mathfrak{C}_{\square_c}(S)(x,y)_n$ may be represented by a pair $(r: R \to S, \sigma_R)$ where 
\begin{itemize}
\item $R$ is a necklace with $|V_R-J_R|=n$ such that there are no  $(u: U \to S) \in (Nec \downarrow S)_{x,y}$ with $|V_U -J_U|=n-1$ and $f: R \to U$ satisfying $r=u \circ f$, and
\item $\sigma_R \in C_{\square_c}(R)_n$ is the unique non-degenerate $n$-cell in $C_{\square_c}(R)$.
\end{itemize}
In fact, one can let $R=T_{\sigma}$ and $r=t \circ \iota_{\sigma}$ as in Remark 4.6. These representatives are not unique since we may have another representative $(r': R' \to S,\sigma_{R'})$ if there is a morphism of necklaces $h: R \to R'$ of type (iii) such that $r' \circ h=r$. We write $[r: R \to S]$ for the equivalence class of the non-degenerate $n$-cell  in $\mathfrak{C}_{\square_c}(S)(x,y)$ represented by $(r: R\to S, \sigma_R)$. Let $v$ be the $j$-th vertex in $V_R-J_R$. The face map $\partial^1_j: \mathfrak{C}_{\square_c}(S)(x,y)_n \to \mathfrak{C}_{\square_c}(S)(x,y)_{n-1}$ is given by $\partial^1_j[r: R\to S]= [\partial^1_jr: R_{v} \to S]$ where $R_{v}$ is the subnecklace of $R$ spanned by vertices $V_R-\{v\}$ and $\partial^1_jr$ is the restriction of $r$ to $R_{v}$. The face map $\partial^0_j: \mathfrak{C}_{\square_c}(S)(x,y)_n \to \mathfrak{C}_{\square_c}(S)(x,y)_{n-1}$ is given by $\partial^0_j[r: R \to S]= [\partial^0_jr: R(\alpha_R,v) \vee R(v,\omega_R) \to S]$ where $\partial^0_jr$ is the restriction of $r$ to $R(\alpha_R,v) \vee R(v,\omega_R)$. Of course $[\partial^1_jr: R_{v} \to S]$ and $[\partial^0_jr: R(\alpha_R,v) \vee R(v,\omega_R)\to S]$ may be degenerate cells in $\mathfrak{C}_{\square_c}(S)(x,y)_{n-1}$ even if $[r: R\to S]$ is non-degenerate. 
 \end{remark}
 
Let us recall Lurie's construction of $\mathfrak{C}: Set_{\Delta} \to Cat_{\Delta}$. Given integers $0 \leq  i \leq j$ denote by $P_{i,j}$ the category whose objects are subsets of the set $\{i, i+1, ..., j\}$ containing both $i$ and $j$ and morphisms are inclusions of sets. We have an isomorphism of categories $P_{i,j} \cong \mathbf{1}^{j-i-1}$ if $i<j$ and $P_{i,i} \cong \mathbf{1}^0$. For each integer $n \geq 0$ define a simplicial category $\mathfrak{C}(\Delta^n)$ whose objects are the elements of the set $\{0, ... , n\}$ and for any two objects $i$ and $j$ such that $i \leq j$, $\mathfrak{C}(\Delta^n)(i,j)$ is the simplicial set $N(P_{i,j})$, where $N: Cat \to Set_{\Delta}$ is the nerve functor. If $j < i$, $\mathfrak{C}(\Delta^n)(i,j)$ is defined to be empty. The composition law in the simplicial category $\mathfrak{C}(\Delta^n)$ is induced by the map of categories $P_{j,k} \times P_{i,k} \to P_{i,k}$ given by union of sets. The construction of $\mathfrak{C}(\Delta^n)$ is functorial with respect to simplicial maps between standard simplices. Then the functor $\mathfrak{C}: Set_{\Delta} \to Cat_{\Delta}$ is defined by $\mathfrak{C}(S):= \text{colim}_{\Delta^n \to S} \mathfrak{C}(\Delta^n)$.
\\ 

$\mathfrak{C}$ is defined as a colimit in the category of simplicial categories. Dugger and Spivak computed in \cite{DS11} the mapping spaces of $\mathfrak{C}$ explicitly via necklaces. More precisely, Proposition 4.3 of \cite{DS11} states that there is an isomorphism of simplicial sets
\begin{eqnarray*}
\underset{T \to S \in (Nec \downarrow S)_{x,y}}{\text{colim}} [ \mathfrak{C}(T)(\alpha_T, \omega_T)] \cong \mathfrak{C}(S)(x,y).
\end{eqnarray*}
We defined $\mathfrak{C}_{\square_c}$ having this formula in mind. We do it this way, as opposed to first defining $\mathfrak{C}_{\square_c}$ on standard simplices and then extending as a left Kan extension, to emphasize that maps of necklaces give rise to maps of cubical sets with connections and the relationship of this fact with Adams' cobar construction, as we will explain later on. The mapping spaces of the functor $\mathfrak{C}_{\square_c}$ are cubical sets with connections constructed by applying the Yoneda embedding to the category $P_{\mathbf{1}}(T)$ associated to a necklace $T$ and then taking a colimit, while the mapping spaces in $\mathfrak{C}$ are simplicial sets obtained by applying the nerve functor to $P_{\mathbf{1}}(T)$ and then taking a colimit. 
\\
 
Recall we have a triangulation functor $| \cdot | : Set_{\square_c} \to Set_{\Delta}$ defined on a cubical set with connections $K$ by $|K|:=\text{colim}_{\square_c^n \to K} N(\mathbf{1}^n) \cong \text{colim}_{\square_c^n \to K}  (\Delta^1)^{\times n}$. Define a functor $\mathfrak{T}: Cat_{\square_c} \to Cat_{\Delta}$ as follows. Given a category $\mathcal {K}$ enriched over $Set_{\square_c}$ define $\mathfrak{T}(\mathcal{K})$ to be the simplicial category whose objects are the objects of $\mathcal{K}$ and whose mapping spaces are given by $|\mathcal{K}(x,y)|$ for any objects $x$ and $y$ in $\mathcal{K}$. We have a composition law on $\mathfrak{T}(\mathcal{K})$ induced by applying the functor $| \cdot |$ to the composition law in $\mathcal{K}$ and using the fact that for cubical sets with connections $K$ and $K'$ we have a natural isomorphism $|K\otimes K'| \cong |K| \times |K'|$. In fact, since colimits commute we have the following isomorphisms of simplicial sets
\begin{eqnarray*}
| K \otimes K' | \cong  | \underset{ \square^n_c \to K, \square^m_c \to K'}{\text{colim}} \square_c^{n+m} | \cong  \underset{ \square^n_c \to K, \square^m_c \to K'}{\text{colim}} | \square_c^{n+m} | \cong \underset{ \square^n_c \to K, \square^m_c \to K'}{\text{colim}} (\Delta^1)^{\times n+m} 
\\
\cong \underset{ \square^n_c \to K, \square^m_c \to K'}{\text{colim}} (\Delta^1)^{\times n} \times (\Delta^1)^{\times m} \cong  \underset{ \square^n_c \to K }{\text{colim}} (\Delta^1)^{\times n} \times \underset{ \square^m_c \to K'}{\text{colim}} (\Delta^1)^{\times m} \cong |K| \times |K'|.
\end{eqnarray*} 

\begin{proposition} The functor $\mathfrak{C}: Set_{\Delta} \to Cat_{\Delta}$ is naturally isomorphic to the composition of functors
\begin{eqnarray*}
Set_{\Delta} \xrightarrow{\mathfrak{C}_{\square_c}} Cat_{\square_c} \xrightarrow{\mathfrak{T}} Cat_{\Delta}.
\end{eqnarray*}
\end{proposition}
\begin{proof}
Let $Y(\square_c) \downarrow \square_c^N$ be the category whose objects are morphisms $\square_c^n \to \square_c^N$ of cubical sets with connections and whose morphisms are given by the corresponding commutative triangles.  Note $|\square_c^N|$ is the colimit in simplicial sets of the functor $Y(\square_c) \downarrow \square_c^N \to Set_{\Delta}$ that sends an object $(\square_c^n \to \square_c^N)$ to $N(\mathbf{1}^n)\cong (\Delta^1)^{\times n}$ and a morphism in $Y(\square_c) \downarrow \square_c^N$ to the corresponding induced morphism between nerves. The identity morphism $\square_c^N \to \square_c^N$  is a terminal object in $Y(\square_c) \downarrow \square_c^N$. Therefore, $|\square_c^N|=\text{colim}_{\square_c^n \to \square_c^N} N(\mathbf{1}^n)$ is given by the value of the functor on the identity morphism $\square_c^N \to \square_c^N$, so $|\square_c^N|=N(\mathbf{1}^N)$. 
\\

Let $S$ be a simplicial set. The objects of the simplicial categories $\mathfrak{T}(\mathfrak{C}_{\square_c}(S))$ and $\mathfrak{C}(S)$ are the same, i.e. the elements of $S_0$. Since the triangulation functor $| \cdot |$ commutes with colimits, we have the following natural isomorphisms 
\begin{eqnarray*}
(\mathfrak{T}(\mathfrak{C}_{\square_c}(S)))(x,y) \cong  \underset{T \to S \in (Nec \downarrow S)_{x,y}} {\text{colim}} |C_{\square_c}(T)| 
\cong  \underset{T \to S \in (Nec \downarrow S)_{x,y}} {\text{colim}} N(\mathbf{1}^{|V_T-J_T|}).
\end{eqnarray*}
Moreover, by Proposition 4.3 of \cite{DS11} it follows that we have natural isomorphisms
\begin{eqnarray*}
\underset{T \to S \in (Nec \downarrow S)_{x,y}} {\text{colim}} N(\mathbf{1}^{|V_T-J_T|}) \cong \underset{T \to S \in (Nec \downarrow S)_{x,y}} {\text{colim}}[\mathfrak{C}(T)(\alpha,\omega)]  \cong \mathfrak{C}(S)(x,y).
\end{eqnarray*}
Hence, we have an isomorphism of simplicial categories $\mathfrak{T}(\mathfrak{C}_{\square_c}(S)) \cong \mathfrak{C}(S)$ which is functorial on $S$. It follows that $\mathfrak{T} \circ \mathfrak{C}_{\square_c}$ and $\mathfrak{C}$ are naturally isomorphic functors. 
\end{proof}

\section{The left adjoint $\Lambda: Set_{\Delta} \to dgCat_k$ of the DG nerve functor}

In section 1.3.1 of \cite{Lur11} Lurie defines a functor $N_{dg}: dgCat_k \to Set_{\Delta}$, called the \textit{dg nerve}, which is weakly equivalent to the left adjoint of the composite functor
\begin{eqnarray*}
\Gamma: Set_{\Delta} \xrightarrow{\mathfrak{C}} Cat_{\Delta} \xrightarrow{\mathfrak{Q}_{\Delta}} dgCat_k
\end{eqnarray*}
where $\mathfrak{Q}_{\Delta}$ is the functor obtained by applying the normalized chains functor $Q_{\Delta}: Set_{\Delta} \to Ch_k$ on the mapping spaces. In this section we prove that the composite functor
\begin{eqnarray*}
\Lambda: Set_{\Delta} \xrightarrow{\mathfrak{C}_{\square_c}} Cat_{\square_c} \xrightarrow{\mathfrak{Q}_{\square^c}} dgCat_k,
\end{eqnarray*}
where $\mathfrak{Q}_{\square^c}$ is the functor obtained by applying the normalized chains functor $Q_{\square_c}: Set_{\square_c} \to Ch_k$ on the mapping spaces, is left adjoint to $N_{dg}$. 
\\

Recall Lurie's definition of $N_{dg}$. Let $\mathcal{C}$ be a dg category.  For each $n\geq0$, define $N_{dg}(\mathcal{C})_n$ to be the set of all ordered pairs of sets $(\{X_i \}_{0 \leq i \leq n}, \{f_I\})$, such that:
\begin{enumerate}
\item $X_0,X_1,...,X_n$ are objects of the dg category $\mathcal{C}$
\item $I$ is a subset $I =\{i_{-}< i_m < i_{m-1} < ... < i_1< i_{+} \} \subseteq [n]$ with $m \geq 0$ and $f_I$ is an element of  $\mathcal{C}(X_{i_-}, X_{i_+})_m$ satisfying
\begin{eqnarray*}
df_{I}=\sum_{1 \leq j \leq m} (-1)^j(f_{I-\{i_j\}}- f_{i_j< ...<i_1<i_+} \circ f_{i_-< i_m < ...<i_j}).
\end{eqnarray*}
\end{enumerate}
The structure maps in $N_{dg}(C)$ are defined as follows. If  $\alpha: [m] \to [n]$ is a nondecreasing function, then the induced map $N_{dg}(\mathcal{C})_n \to N_{dg}(\mathcal{C})_m$ is given by
\begin{eqnarray*}
(\{X_i\}_{0 \leq 1 \leq n}, \{f_I\}) \mapsto (\{X_{\alpha(j)} \}_{0 \leq j \leq m}, \{g_J\}),
\end{eqnarray*}
where $g_J= f_{\alpha(J)}$ if $\alpha|_{J}$ is injective, $g_J=id_{X_i}$ if $J=\{j,j'\}$ with $\alpha(j)=i=\alpha(j')$, and $g_J=0$ otherwise. 
\begin{theorem}
The functor $\Lambda: Set_{\Delta} \to dgCat_k$ is left adjoint to $N_{dg}: dgCat_k \to Set_{\Delta}$.
\end{theorem}
\begin{proof}
First, we show that for any standard simplex $\Delta^n$ and any dg category $\mathcal{C}$ there is bijection 
\begin{eqnarray*}
\theta_{n,\mathcal{C}}: dgCat_k(\Lambda(\Delta^n), \mathcal{C}) \cong Set_{\Delta}(\Delta^n, N_{dg}(\mathcal{C}))
\end{eqnarray*}
 which is functorial with respect to morphisms in the category $\Delta$. Given a dg functor $F: \Lambda(\Delta^n) \to \mathcal{C}$ we construct an $n$-simplex $$\theta_{n,\mathcal{C}} (F)=(\{X_0,...,X_n\}, \{f_I\})$$ in $N_{dg}(\mathcal{C})_n$. The objects of $\Lambda(\Delta^n)$ are the integers $0,1,...,n$ so we let $X_i=F(i)$ for $i=0,1,...,n$. For every subset $I=\{i_- < i_1 < ... < i_m <i_+ \} \subseteq [n]$ define $\sigma_I$ to be the generator of the chain complex $\Lambda(\Delta^n)(i_-, i_+)=Q_{\square_c}(\mathfrak{C}_{\square_c}(\Delta^n)(i_-,i_+))$ represented by the non-degenerate element of $(\mathfrak{C}_{\square_c}(\Delta^n)(i_-,i_+))_m$ which is the one bead sub-necklace inside $\Delta^n$ consisting of the $(m+1)$-simplex with $i_-$ as first vertex, $i_+$ as last vertex, and $i_1,...,i_m$ as non-joint vertices, in other words, $\sigma_I$ is represented by the $(m+1)$-simplex inside $\Delta^n$ spanned by vertices $i_-,i_1,...,i_m,i_+$. It follows from Remark 3.2 that
 \begin{eqnarray*}
 d\sigma_I=\sum_{j=1}^m(-1)^j(\partial^1_j \sigma_I - \partial^0_j \sigma_I)=\sum_{j=1}^m (-1)^j( \sigma_{I-\{i_j\}} -\sigma_{i_j< ...<i_1<i_+} \circ \sigma_{i_-< i_m < ...<i_j}).
 \end{eqnarray*}
Define $f_I=F(\sigma_I): X_{i_-} \to X_{i_+}$. Since the dg functor $F$ commutes with differentials at the level of mapping spaces, $f_I$ satisfies property (2) in the definition of the dg nerve functor. The functoriality of $\theta_{n,\mathcal{C}}$ with respect to simplicial maps between standard simplices follows from Proposition 4.2. Finally, since the functor $\Lambda$ preserves colimits, $\theta_{n,\mathcal{C}}$ induces a functorial bijection 
\begin{eqnarray*}
dgCat_k(\Lambda(S), \mathcal{C} ) \cong Set_{\Delta}(S, N_{dg}(\mathcal{C}))
\end{eqnarray*}
for any simplicial set $S$ and dg category $\mathcal{C}$.
\end{proof}

\begin{remark} Let $S$ be a simplicial set and $x,y \in S_0$. A generator $\xi$ of degree $n$ in the chain complex $\Lambda(S)(x,y)$ is an equivalence class which may be represented by a non-degenerate $n$-cell $\sigma$ in the cubical set with connections $\mathfrak{C}_{\square_c}(S)(x,y)$. Since $\mathfrak{C}_{\square_c}(S)(x,y)$ is defined as a colimit, the non-degenerate $n$-cell $\sigma$ is itself an equivalence class $[r: \Delta^{n_1} \vee ... \vee \Delta^{n_k} \to S]$, where $(r: \Delta^{n_1} \vee ... \vee \Delta^{n_k} \to S)  \in (Nec \downarrow S)_{x,y}$, $n_1+...+n_k-k=n$ and such that there is no $(u: \Delta^{m_1} \vee ... \vee \Delta^{m_l}  \to S ) \in (Nec \downarrow S)_{x,y}$ with $m_1 + ...+m_l - l < n$ together with a map of necklaces 
\begin{eqnarray*}
f: \Delta^{n_1} \vee ... \vee \Delta^{n_k} \to \Delta^{m_1} \vee ... \vee \Delta^{m_l}
\end{eqnarray*}
satisfying $r=u \circ f$. Moreover, any 
\begin{eqnarray*}
s: \Delta^{n_1} \vee ... \vee \Delta^{n_i} \vee \Delta^1 \vee \Delta^{n_{i+1}} \vee ... \vee \Delta^{n_k} \to S
\end{eqnarray*}
satisfying $r \circ  \pi=s$, where $\pi: \Delta^{n_1} \vee ... \vee \Delta^{n_i} \vee \Delta^1 \vee \Delta^{n_{i+1}} \vee ... \vee \Delta^{n_k} \to \Delta^{n_1} \vee ... \vee \Delta^{n_k}$ is the map of simplicial sets which collapses the $(i+1)$-th bead in the domain necklace to a point, also represents the equivalence class $\sigma$. This follows essentially from Proposition 4.2 (3).
 \end{remark}

\section{Rigidification and the cobar construction}

In this section, we relate the functor $\mathfrak{C}_{\square_c}: Set^0_{\Delta} \to Cat_{\square_c}$ to the cobar functor $\Omega: dgCoalg^0_k \to dgAlg_k$. More precisely, we prove that  $\Omega Q_{\Delta}(S)$, the cobar construction on the dg coalgebra of normalized chains on a simplicial set $S$ with one vertex $x$, is isomorphic as a dga to $\Lambda(S)(x,x)$, where $\Lambda$ is the functor obtained by applying the normalized cubical chains functor on the mapping spaces of $\mathfrak{C}_{\square_c}$, or naturally isomorphically, the left adjoint to the dg nerve functor, as described in the previous section. Then we deduce a relationship between $\mathfrak{C}: Set^0_{\Delta} \to Cat_{\Delta}$ and $\Omega: dgCoalg^0_k \to dgAlg_k$: we show $\Omega Q_{\Delta}(S)$ is naturally weakly equivalent (quasi-isomorphic) as a dga to $\Gamma (S) (x,x)$, where $\Gamma: Set_{\Delta} \to dgCat$ is the functor obtained by applying normalized chains to the mapping spaces of $\mathfrak{C}$.
\\

Let $k$ be a fixed commutative ring. We may consider $k$ as a graded $k$-module concentrated on degree $0$.  A graded coassociative coalgebra $(C, \Delta)$ over $k$ is \text{counital} if it is equipped with a degree $0$ map $\epsilon: C \to k$, called the \textit{counit}, such that $(\epsilon \otimes \text{id} ) \circ \Delta= \text{id} = (\text{id} \otimes \epsilon) \circ \Delta$. 

We say a differential graded coassociative coalgebra (dg coalgebra, for short) $(C,\partial, \Delta)$ over a commutative ring $k$ is \text{connected} if $C_0\cong k$. Given a connected dg coalgebra $(C,\partial, \Delta)$ which is free as a $k$-module in each degree, the \textit{cobar construction} of $C$ is the differential graded associative algebra $(\Omega C,D)$ defined as follows. Consider the graded $k$-module $s\overline{C}$ where $\overline{C}_i=C_i$ for $i>0$ and $\overline{C}_0=0$ and $s$ is the shift by $-1$, i.e. $(s\overline{C})_i=\overline{C}_{i+1}$. Let $\Delta= \text{Id} \otimes 1 + 1 \otimes \text{Id} + \Delta'$ and for any $c \in \overline{C}$ write $\Delta'(c)= \sum c' \otimes c''$. The underlying algebra of the cobar construction is the tensor algebra 
$$\Omega C= Ts\overline{C}= k \oplus s\overline{C} \oplus (s\overline{C} \otimes s\overline{C}) \oplus (s\overline{C} \otimes s\overline{C} \otimes s\overline{C}) \oplus ...$$
 and the differential $D$ is defined by extending $D(sc) =-s\partial c + \sum (-1)^{\text{deg } c'} sc' \otimes sc''$ as a derivation to all of $\Omega C$. This construction yields a functor $\Omega: dgCoalg^0_k \to dgAlg_k$ where $dgAlg_k$ is the category of augmented dg algebras over $k$. 
\\

For any simplicial set $S$, the chain complex $Q'_{\Delta}(S)$ of \textit{unnormalized} chains over $k$ has a natural coproduct $\Delta: Q'_{\Delta}(S) \to Q'_{\Delta}(S) \otimes Q'_{\Delta}(S)$ given by
\begin{eqnarray*}
\Delta(x) = \bigoplus_{p+q=n}f^p( x) \otimes l^q(x)
\end{eqnarray*}
for any $x  \in Q_{\Delta}(S)_n$, where $f^p$ denotes the \textit{front $p$-face map} (induced by the map $[p] \to [p+q]$, $i \mapsto i$) and $l^q$ is the \textit{last $q$-face map} (induced by the map $[q] \to [p+q]$, $i \mapsto i+p$). This coproduct is known as the Alexander-Whitney diagonal map. Moreover, this dg coalgebra structure passes to the \textit{normalized} chain complex $Q_{\Delta}(S)$. Thus, we may consider $Q_{\Delta}$ as a functor $Q_{\Delta}: Set_{\Delta} \to dgCoalg_k$. In particular, $Q_{\Delta}(S)$ is a dg coalgebra which is free as a $k$-module in each degree. If $S$ is $0$-reduced, i.e. $S_0=\{x\}$, then $Q_{\Delta}(S)$ is counital and connected with counit map given by the composition $Q_{\Delta}(S)  \twoheadrightarrow Q_{\Delta}(S)_0= k[x] \xrightarrow{\cong} k$.  From now on all of the coalgebras in this article will be assumed to be counital.  
  
 \begin{theorem} \label{7.1} Let  $S$ be a $0$-reduced simplicial set with $S_0=\{x\}$. There is an isomorphism of differential graded algebras $\Lambda(S)(x,x) \cong \Omega Q_{\Delta}(S)$.
 \end{theorem}
 
 \begin{proof}  For each integer $n \geq 0$ the boundary map  $\partial: Q'_{\Delta} (S)_n \to Q'_{\Delta}(S)_{n-1} $ and the coproduct $\Delta: Q'_{\Delta} (S)_n \to \bigoplus_{p+q=n} Q'_{\Delta}(S)_p\otimes Q'_{\Delta}(S)_q$ can be written as sums $\partial= \sum_{i=0}^{n} (-1)^i \partial_i$ and $\Delta=\sum_{i=0}^n  \Delta_i$ as usual. In particular, for $\sigma \in S_n$,  $\Delta_0(\sigma)= \text{min}\sigma \otimes \sigma$ and $\Delta_n{\sigma}= \sigma \otimes \text{max}\sigma$ where $\text{min}{\sigma}$ and $\text{max}\sigma$ denote the first and last vertices of $\sigma$, respectively. The truncated maps $\partial'=\sum_{i=1}^{n-1} (-1)^i \partial_i$ and  $\Delta'=\sum_{i=1}^{n-1} (-1)^i \Delta_i$ also define a differential graded coassociative coalgebra structure on $Q'_{\Delta}(S)$. Consider the dga $\Omega Q'_{\Delta}(S) =  \Omega  (Q'_{\Delta}(S), \partial' , \Delta')$. First, we show $\Lambda(S)(x,x)= Q_{\square_c}(\mathfrak{C}_{\square_c}(S)(x,x)) \cong \Omega  Q'_{\Delta}(S) /\sim$ for some equivalence relation $\sim$ and then we construct an isomorphism 
 \begin{eqnarray*}
 \Omega  Q'_{\Delta}(S) /\sim \text{} \cong  \Omega Q_{\Delta}(S).
 \end{eqnarray*}
The dga $\Omega  Q'_{\Delta}(S)  $ has as underlying complex the tensor algebra $Ts\overline{Q'_{\Delta}(S)}$ together with differential $D'_{\Omega}=\partial' + \Delta'$ extended as a derivation to all of $Ts\overline{Q'_{\Delta}(S)}$. We denote a monomial $s\sigma_1 \otimes ... \otimes s\sigma_k \in Ts\overline{Q'_{\Delta}(S)}$ by $[\sigma_1|...|\sigma_k]$. Let $s_0(x) \in Q'_{\Delta}(S)_1$ be the generator corresponding to the degenerate $1$-simplex at $x$. We take a quotient of $Ts\overline{Q'_{\Delta}(S)}$  by the equivalence relation generated by
 \begin{eqnarray*}
 [\sigma_1 |... | \sigma_k ]\sim [\sigma_1 | ... |\sigma_{i-1}  | \sigma_{i+1} | ... |\sigma_k]
 \end{eqnarray*}
if for some $1 \leq i \leq k$ we have $\sigma_i=s_0(x)$ (in particular $[\sigma_1] \sim 1_k$ if $\sigma_1=s_0(x)$); and
 \begin{eqnarray*}
 [\sigma_1 | ... | \sigma_k ]\sim 0
 \end{eqnarray*}
 if $\sigma_i \in Q'_{\Delta}(S)_{n_i}$ is a degenerate simplex with $n_i > 1$ for some $1 \leq i \leq k$. The first relation corresponds to the identification in the colimit defining $\mathfrak{C}_{\square_c}(S)(x,x)$ arising from Remark 4.6; the second relation corresponds to modding out by degenerate chains in the definition of the normalized chain complex $Q_{\square_c}(\mathfrak{C}_{\square_c}(S)(x,x))$. Both the differential $D'_{\Omega}$ and the algebra structure of  $Ts\overline{Q'_{\Delta}(S)}$ pass to the quotient
 \begin{eqnarray*}
 Ts\overline{Q'_{\Delta}(S)}/\sim.
 \end{eqnarray*}  
It is clear that we have an isomorphism of dga's
 \begin{eqnarray*}
 Q_{\square_c}(\mathfrak{C}_{\square_c}(S) (x,x)) \cong \Omega Q'_{\Delta}(S) /\sim
 \end{eqnarray*} 
 since necklaces in $S$ correspond to monomials of generators in $Q'_{\Delta}(S)$.
 \\
 
We define an isomorphism of dga's 
\begin{eqnarray*}
\tilde{\varphi}: \Omega  Q'_{\Delta}(S) /\sim \text{} \to \Omega Q_{\Delta}(S).
\end{eqnarray*}
Given $\sigma \in Q'_{\Delta}(S)$ denote by $\overline{\sigma}$ the equivalence class of $\sigma$ in  $Q_{\Delta}(S)$. First define $\varphi  [\sigma ] = [\overline{\sigma}]$ if $\text{deg} \sigma >1$, $\varphi  [\sigma ] = \overline{\sigma} +1_k$ if $\text{deg} \sigma = 1$, and $\varphi(1_k)=1_k$. Extend $\varphi$ as an algebra map to obtain a map $\varphi: \Omega Q'_{\Delta}(S) \to \Omega Q_{\Delta}(S)$. It follows by a short computation that the map $\varphi$ is a chain map. Moreover, $\varphi$ induces a map of dga's $\tilde{\varphi}: \Omega  Q'_{\Delta}(S) /\sim \text{} \to \Omega Q_{\Delta}(S)$. The map $\tilde{\varphi}$ is an isomorphism  of dga's, in fact, the inverse map $\psi:  \Omega Q_{\Delta}(S) \to \Omega Q'_{\Delta}(S) /\sim$ is given by defining $\psi[\overline{\sigma}]= \big[ [\sigma] \big]$ if $\text{deg} \sigma >1$, $\psi [\overline{\sigma}]= \big[ [\sigma] \big] -\big[1_k\big]$ if $\text{deg}\sigma=1$, and $\psi(1_k)=\big[ 1_k \big]$ and then extending $\psi$ as an algebra map, where $\big[ [\sigma] \big]$ denotes the equivalence class of $[\sigma] \in \Omega Q'_{\Delta}$ in the quotient $\Omega Q'_{\Delta}(S) /\sim$.  
\end{proof}
 
We now relate the dga's $\Omega Q_{\Delta}(S)$ and $\Gamma(S)(x,x)$. We will use the following lemma which follows from an acyclic models argument. 

\begin{lemma} \label{7.2}
For any cubical set with connections $K$ the chain complex $Q_{\Delta}(|K|)$ is naturally weakly equivalent to $Q_{\square_c}(K)$, where $| \cdot |: Set_{\square_c} \to Set_{\Delta}$ is the triangulation functor. 
\end{lemma}
\begin{proof}
This proposition follows from the Acyclic Models Theorem applied to the two functors 
\begin{eqnarray*}
Q_{\Delta} \circ | \cdot |,  Q_{\square_c}: Set_{\square_c} \to Ch_k.
\end{eqnarray*}
Define the collection of models in $Set_{\square_c}$ to be $\mathcal{M}=\{\square^0_c, \square^1_c, ... \}$, where $\square^j_c$ is the standard $j$-cube with connections. It is clear that both $Q_{\Delta} \circ | \cdot |$ and $Q_{\square_c}$ are acyclic on these models. Recall a functor $F: \mathcal{C} \to Ch_k$ is \textit{free} on $\mathcal{M}$ if there exist a collection $\{M_j\}_ {j \in \mathcal{J}}$ where each $M_j$ is an object in $\mathcal{M}$ (possibly with repetitions, possibly not including all of the objects in $\mathcal{M}$) together with elements $m_j \in F(M_j)$ such that for any object $X$ of $\mathcal{C}$ we have that $\{F(f)(m_j) \in F(X)| j \in \mathcal{J}, (f: M_j \to X) \in \mathcal{C}(M_i,X)\}$ forms a basis for $F(X)$. Clearly $Q_{\square_c}$ is free on $\mathcal{M}$ since we can take $M_j=\square^j_c, \mathcal{J}=\{0,1,2,...,\}$, and define $m_j \in Q_{\square_c}(M_j)=Q_{\square_c}(\square^j_c)$ to be the generator corresponding to the unique non-degenerate element in $(\square^j_c)_j$ (i.e. $m_j$ is the top non-degenerate cell of $\square^j_c$). Note that the simplicial set $|\square^j_c| \cong (\Delta^1)^{\times j}$ has $j!$ non-degenerate $j$-simplices $\sigma^j_1,...,\sigma^j_{j!} \in |\square^j_c|_j$. Hence, $Q_{\Delta} \circ | \cdot |$ is also free on $\mathcal{M}$ since we can take $\{M^0_1, M^1_1, M^2_1, M^2_2, ..., M^j_1,...,M^j_{j!}, M^{j+1}_1,... \}_{j \in \mathcal{J}}$ where $M^j_k=\square^j_c, \mathcal{J}=\{0,1,2,...\}$, and $m^j_k \in Q_{\Delta}(|M^j_k|)$ the generator corresponding to the $j$-simplex $\sigma^j_k \in |\square^j_c|_j$. 
\\

We have a natural isomorphism of functors $H_0(Q_{\Delta} \circ | \cdot |) \cong H_0(Q_{\square_c})$, in fact, for any $K \in Set_{\square_c}$ there is a natural bijection between $|K|_0$ and $K_0$ and any two vertices $x$ and $y$ are connected by a sequence of $1$-simplices in $|K|_1$ if and only if they are connected by a sequence of $1$-cubes in $K_1$. By the Acyclic Models Theorem there exist natural transformations $\phi: Q_{\Delta} \circ | \cdot | \to Q_{\square_c}$ and $\psi: Q_{\square_c} \to Q_{\Delta} \circ | \cdot |$ such that each composition $\phi \circ \psi$ and $\psi \circ \phi$ is chain homotopic to the identity map. 
\end{proof}
We use the above lemma to relate $\Omega Q_{\Delta}(S)$ and $\Gamma(S)(x,x)$. 
 \begin{proposition} \label{7.3}
 Let $S$ be a $0$-reduced simplicial set with $S_0=\{x\}$. The differential graded associative algebras $\Omega Q_{\Delta}(S)$ and $\Gamma(S)(x,x)$ are naturally weakly equivalent.
 \end{proposition}
 \begin{proof}
 By Theorem \ref{7.1} we have an isomorphism 
 \begin{eqnarray*}
 \Omega Q_{\Delta} (S) \cong \Lambda(S)(x,x) = Q_{\square_c}(\mathfrak{C}_{\square_c}(S)(x,x)).
 \end{eqnarray*}
 By Lemma \ref{7.2} and the fact that the triangulation functor and chains functor preserve the monoidal structures, it follows that the dga's $Q_{\square_c}(\mathfrak{C}_{\square_c}(S)(x,x))$ and $Q_{\Delta}|\mathfrak{C}_{\square_c}(S)(x,x)|$ are naturally weakly equivalent. Finally, note that we have isomorphisms  
  \begin{eqnarray*}
  Q_{\Delta}|C_{\square_c}(S)(x,x)| = Q_{\Delta}((\mathfrak{T} \circ \mathfrak{C}_{\square_c})(S)(x,x)) \cong Q_{\Delta}(\mathfrak{C}(S)(x,x))=\Gamma(S)(x,x).
  \end{eqnarray*} 
 \end{proof}

\section{Properties of $\mathfrak{C}: Set_{\Delta} \to Cat_{\Delta}$ }

We recall several homotopy theoretic properties of the rigidification functor $\mathfrak{C}: Set_{\Delta} \to Cat_{\Delta}$, in particular, its behavior with respect to Kan weak equivalences  and its relationship with path spaces. These will be used in the final section of the article.
\\

A map of simplicial sets $f: S \to S'$ is called a \textit{Kan weak equivalence} if it is a weak equivalence in the Quillen model structure, namely, if $f$ induces a weak homotopy equivalence of spaces $|f|: |S| \to |S'|$. A map of simplicial sets $f:S \to S'$ is called a \textit{categorical equivalence} if $f$ induces a weak equivalence $\mathfrak{C}(f): \mathfrak{C}(S) \to \mathfrak{C}(S')$ of simplicial categories in the Bergner model structure. Recall that a functor of simplicial categories $F: \mathcal{C} \to \mathcal{C}'$ is called a \textit{weak equivalence of simplicial categories} if
\begin{itemize}
\item $F$ induces an essentially surjective functor at the level of homotopy categories, and
\item for all $x,y \in \mathcal{C}$, $F: \mathfrak{C}(S)(x,y) \to  \mathcal{C}'(F(x), F(y))$ is a Kan weak equivalence of simplicial sets. 
\end{itemize}

The Quillen model structure on $Set_{\Delta}$ has Kan equivalences as weak equivalences and Kan complexes as fibrant objects. There is a different model structure on $Set_{\Delta}$, the Joyal model structure, which has categorical equivalences as weak equivalences and quasi-categories as fibrant objects. Moreover, the Quillen model structure is a left Bousfield localization of the Joyal model structure. In particular, a categorical equivalence is always a Kan weak equivalence. The converse is not true in general, but a Kan weak equivalence between Kan complexes is always a categorical equivalence. This is Proposition 17.2.8 in \cite{Rie14}, which we record below.

\begin{proposition} \label{Kan} If $f: S \to S'$ is a Kan weak equivalence between Kan complexes $S$ and $S'$ then $\mathfrak{C}(f): \mathfrak{C}(S) \to \mathfrak{C}(S')$ is a weak equivalence of simplicial categories.  
\end{proposition}

A map $f: C \to C'$ of connected dg coalgebras is called a \textit{quasi-isomorphism} if $f$ induces an isomorphism of coalgebras after passing to homology. On the other hand, a map $f: C \to C'$ of connected dg coalgebras is called an \textit{$\Omega$-quasi-isomorphism} if $f$ induces a quasi-isomorphism of dga's $\Omega f: \Omega C \to \Omega C'$.  An $\Omega$-quasi-isomorphism between connected dg coalgebras is always a quasi-isomorphism. The converse is not true in general, namely, a quasi-isomorphism between connected dg coalgebras might not be an $\Omega$-quasi-isomorphism. However, if $C$ and $C'$ are connected dg coalgebras which are \textit{simply connected} (i.e. $C_1=0=C'_1$) then a quasi-isomorphism $f: C\to C'$ is an $\Omega$-quasi-isomorphism. This follows by comparing Eilenberg-Moore spectral sequences. There are model structures of the category of connected dg coalgebras having each of these two notions as the weak equivalences, but we do not need these for the purposes of this paper. 
\\

Let $Set_{\Delta}^0$ be the full subcategory of the category $Set_{\Delta}$ of simplicial sets whose objects are $0$-reduced simplicial sets. Let $dgCoalg_k^0$ be the full subcategory of the category $dgCoalg_k$ of dg coalgebras whose objects are connected dg coalgebras. The normalized chains functor restricts to a functor $Q_{\Delta}: Set_{\Delta}^0 \to dgCoalg_k^0$. 

\begin{proposition}\label{7.4}
The functor $Q_{\Delta}: Set_{\Delta}^0 \to dgCoalg_k^0$ sends Kan weak equivalences to quasi-isomorphisms and categorical equivalences to $\Omega$-quasi-isomorphisms. 
\end{proposition}
\begin{proof}
The proof of the first part of the proposition is well known. For the second suppose $f: S \to S'$ is a categorical equivalence and $S_0=\{x\}$, $S_0'=\{x'\}$. Then we have an induced Kan weak equivalence of simplicial sets $\mathfrak{C}(f): \mathfrak{C}(S)(x,x) \to \mathfrak{C}(S)(x',x')$. This induces a dga quasi-isomorphism 
$Q_{\Delta}\mathfrak{C}(f): Q_{\Delta} ( \mathfrak{C}(S)(x,x)) \to Q_{\Delta} (\mathfrak{C}(S)(x',x'))$. The result follows since the dga's $Q_{\Delta} ( \mathfrak{C}(S)(x,x))$ and $Q_{\Delta} (\mathfrak{C}(S)(x',x'))$ are naturally weakly equivalent to the dga's $\Omega Q_{\Delta}(S)$  and $\Omega Q_{\Delta}(S')$, respectively, by Proposition \ref{7.3}. 
\end{proof}

For any pointed topological space  $(X,b)$ denote by $\text{Sing}(X,b)$ the sub-simplicial set of $\text{Sing}(X)$ whose $n$-simplices are the continuous maps $|\Delta^{n}| \to X$ that take all vertices of $|\Delta^n|$ to $b$. Define a new functor $Q_{\Delta}^K: Set^0_{\Delta} \to dgCoalg^0_k$ by $Q_{\Delta}^K(S):= Q_{\Delta}(\text{Sing}(|S|,x))$, where $S_0=\{x\}$ and $\text{Sing}(|S|,x)$ is the Kan complex of singular simplices $|\Delta^n| \to |S|$ sending all vertices of $|\Delta^n|$ to $x \in |S|$. In general, the functor $Q_{\Delta}$ does not send Kan weak equivalences of simplicial sets to $\Omega$-quasi-isomorphisms, but $Q_{\Delta}^K$ does.

\begin{proposition} \label{derivedchains} The functor  $Q_{\Delta}^K: Set^0_{\Delta} \to dgCoalg^0_k$ sends Kan weak equivalences of simplicial sets to $\Omega$-quasi-isomorphisms of dg coalgebras.
\end{proposition}
\begin{proof}
Let $S, S' \in Set^0_{\Delta}$ with $S_0=\{x\}$ and $S_0'=\{x'\}$. If $f: S \to S'$ is a Kan weak equivalence then $|f|: (|S|,x) \to (|S'|, x')$ is a homotopy equivalence of pointed spaces. The functor $(X,b) \mapsto \text{Sing}(X,b)$ from the category of pointed spaces to $Set^0_{\Delta}$ sends homotopy equivalences of pointed spaces to Kan weak equivalences of $0$-reduced Kan complexes. Thus $\text{Sing}(|f|): \text{Sing}(|S|,x) \to  \text{Sing}(|S'|,x')$ is a Kan weak equivalence. 
It follows from Propositions \ref{Kan} and \ref{7.4} that $Q_{\Delta} (\text{Sing}(|f|)) : Q_{\Delta} ( \text{Sing}(|S|,x)) \to Q_{\Delta}( \text{Sing}(|S'|,x'))$ is an $\Omega$-quasi-isomorphism.
\end{proof}

We now explain the relationship between mapping spaces of $\mathfrak{C}$ and different kinds of spaces of paths in a path connected topological space. This relationship is deduced from the homotopy theoretic properties of $\mathfrak{C}$ as studied in Section 2.2 of \cite{Lur09} and in \cite{DS211} using different methods.
\\

For any simplicial category $\mathcal{C}$  define the \textit{simplicial nerve} $N_{\Delta}(\mathcal{C})$ to be the simplicial set whose set of $n$-simplices is given by
\begin{eqnarray*}
(N_{\Delta}(\mathcal{C}))_n=\text{Hom}_{Cat_{\Delta}}(\mathfrak{C}(\Delta^n),\mathcal{C}).
\end{eqnarray*}
It follows that $N_{\Delta}: Cat_{\Delta} \to Set_{\Delta}$ is the right adjoint of $\mathfrak{C}: Set_{\Delta} \to Cat_{\Delta}$. If $\mathcal{C}$ is a topological category, then the \textit{topological nerve} $N_{Top}(\mathcal{C})$ is defined to be the simplicial nerve of the simplicial category $\text{Sing}(\mathcal{C})$ obtained by applying $\text{Sing}$ to each morphism space of $\mathcal{C}$. As its well known, for any topological monoid $G$, $|N_{Top}(G)|$ is a model for the classifying space $BG$. 
\\

In Section 2.2 of \cite{Lur09}, Lurie shows that the pair of adjoint functors $(\mathfrak{C}, N_{\Delta})$ defines a Quillen equivalence between model categories $Set_{\Delta}$ with the Joyal model structure and $Cat_{\Delta}$ with the Bergner model structure. In particular, for any fibrant simplicial category $\mathcal{C}$ (a simplicial category whose mapping spaces are Kan complexes) the counit map $\mathfrak{C}(N_{\Delta}(\mathcal{C})) \to \mathcal{C}$ is a weak equivalence of simplicial categories. This also follows from Theorem 1.5 of \cite{DS211}. 
\\

Let $X$ be a path connected topological space and let $x,y \in X$. Define the space of Moore paths in $X$ between $x$ and $y$ to be $P^M_{x,y}X= \{ (\gamma, r)| \gamma: [0,\infty) \to X, \gamma(0)=x, \gamma(s)=y \text{ for } r \leq s, r \in [0,\infty)\}$ topologized as a subset of $\text{Map}( [0,\infty), X) \times [0, \infty)$, where $\text{Map}( [0,\infty), X)$ is equipped with the compact-open topology. Define a functor $$\mathcal{P}: Top \to Cat_{Top}$$ from the category of topological spaces to the category of topological categories as follows. For any $X \in Top$ the objects of $\mathcal{P}(X)$ are the points of $X$. For any $x,y \in X$, define the space of morphisms $\mathcal{P}(X)(x,y):= P^M_{x,y}X$ with composition rule induced by concatenation of paths. We call $\mathcal{P}: Top \to Cat_{Top}$ the \textit{path category functor}. 
\\

 The functor $\mathfrak{C}: Set_{\Delta} \to Cat_{\Delta}$ is a simplicial model for the path category functor as shown in Proposition \ref{7.6} below.  Denote by $\text{Sing}(\mathcal{P}X)$ the simplicial category obtained by applying $\text{Sing}$ to the morphism spaces of the topological category $\mathcal{P}X$.
 
\begin{proposition}\label{7.6}
Let $X$ be a path connected topological space. The simplicial categories $\mathfrak{C}(\text{Sing}(X))$ and $\text{Sing}(\mathcal{P}X)$ are weakly equivalent. 
\end{proposition}
\begin{proof}
Choose $b \in X$. The topological category $\mathcal{P}X$ is weakly equivalent to $\Omega X$, the topological category with a single object $b$ and as morphism space $\Omega X (b,b)= \Omega^M_bX$ the space of based Moore loops at $b$ with composition law given by concatenation of loops. A weak equivalence $\mathcal{P}X \to \Omega X$ of topological categories is given by fixing a collection of paths $\mathcal{O}= \{ \gamma_x \}_{x \in X}$ where $\gamma_x$ is a path from $b$ to $x$. More precisely, we have a functor $F_{\mathcal{O}}: \mathcal{P}X \to \Omega X$  given on objects by sending all objects of $\mathcal{P}X$ to the single object of $\Omega X$ and on morphisms $F _{\mathcal{O}}: \mathcal{P}X(x,y) \to \Omega X(b,b)$ is the continuous map $F_{\mathcal{O}}(\gamma)=  \gamma_y^{-1}*  \gamma * \gamma_x$, where $*$ denotes concatenation. The functor $F_{\mathcal{O}}$ is clearly a weak equivalence of topological categories. The topological nerve $N_{Top}= N _{\Delta} \circ \text{Sing}: Cat_{Top} \to Set_{\Delta}$ sends weak homotopy equivalences of topological categories to Kan weak equivalence of simplicial sets. Thus, the simplicial sets $N_{Top}(\mathcal{P}X)$ and  $N_{Top}(\Omega X)$ are Kan weakly equivalent. Moreover, the geometric realization $|N_{Top}(\Omega X)|$ is a model for $B(\Omega X)$, the classifying space of the topological monoid of based loops. It follows from $B(\Omega X) \simeq X$ that the simplicial sets $N_{Top}(\mathcal{P}X)$ and $\text{Sing}(X)$ are Kan weakly equivalent. On the other hand, since the homotopy category of $N_{Top}(\mathcal{P}X)$ is a groupoid it follows that $N_{Top}(\mathcal{P}X)$ is a Kan complex \cite{Joy02}. By Proposition \ref{Kan} we have that $\mathfrak{C}(N_{Top}(\mathcal{P}X))$ and $\mathfrak{C}(\text{Sing}(X))$ are weakly equivalent as simplicial categories. Since $\mathfrak{C} \circ N_{\Delta} (\mathcal{C})  \simeq \mathcal{C}$ for any $\mathcal{C} \in \mathcal{Set}_{\Delta}$ whose mapping spaces are Kan complexes, it follows that $\mathfrak{C} (N_{Top}(\mathcal{P}X)) = \mathfrak{C} (N_{\Delta}( \text{Sing}( \mathcal{P}X))) \simeq \text{Sing}(\mathcal{P}X)$. Hence, the simplicial categories $\mathfrak{C}(\text{Sing}(X))$ and $\text{Sing}(\mathcal{P}X)$ are weakly equivalent. \end{proof}

We have the following corollary.
\begin{corollary}\label{7.7}
Let $X$ be a path connected topological space and $b \in X$. The simplicial categories with one object $\mathfrak{C}(\text{Sing}(X,b))$ and $\text{Sing}(\Omega X) $ are weakly equivalent. 
\end{corollary}
\begin{proof} For path connected $X$ the inclusion $\text{Sing}(X,b) \hookrightarrow \text{Sing}(X)$ is a Kan weak equivalence of Kan complexes, so $\mathfrak{C}(\text{Sing}(X))(b,b) \simeq \mathfrak{C}(\text{Sing}(X,b))(b,b)$. Hence, by Proposition \ref{7.6}, $\mathfrak{C}(\text{Sing}(X,b)) \simeq \text{Sing}(\Omega X)$.
\end{proof}

We finish this section by describing more explicitly the weak equivalence of simplicial sets between $\mathfrak{C}(\text{Sing}(X))(x,y)$ and $\text{Sing}(\mathcal{P}X)(x,y)$ given by Proposition \ref{7.6}. We review this for completeness but it is not strictly necessary to follow Section 9. We follow Chapter $2$ of \cite{Lur09}.
\\

Define a cosimplicial object $J^{\bullet}: \Delta \to (\partial \Delta^1 \downarrow Set_{\Delta})$, by letting $J^n$ to be the quotient of the standard simplex $\Delta^{n+1}$ by collapsing the last face (i.e. the face spanned by  vertices $[0,...,n]$) to a vertex. The quotient simplicial set $J^n$ has exactly two vertices which we denote by the integers $0$ and $n+1$. For any $S \in Set_{\Delta}$ and $x,y \in S_0$, there is a simplicial set $\text{Hom}^R_S(x,y)$ called the \textit{right mapping space} defined by letting $\text{Hom}^R_S(x,y)_n$ be the set of all morphisms of simplicial sets $\varphi: J^n \to S$ such that $\varphi(0)=x$ and $\varphi(n+1)=y$, together with structure face and degeneracy maps defined to coincide with the corresponding structure maps of on $S_{n+1}$. Define a cosimplicial simplicial set $Q^{\bullet}$ by letting $Q^n:=\mathfrak{C}(J^n)(0,n+1)$ and denote by $| - |_{Q^{\bullet}}: Set_{\Delta} \to Set_{\Delta}$ the realization functor associated to $Q^{\bullet}$. Recall Proposition 2.2.4.1 of \cite{Lur09}:

\begin{proposition} \label{7.5} 
Let $S$ be an quasi-category containing a pair of objects $x$ and $y$. There is a natural Kan weak equivalence of simplicial sets
\begin{eqnarray*}
f: |\text{Hom}^R_S(x,y)|_{Q^{\bullet} } \to \mathfrak{C}(S)(x,y).
\end{eqnarray*} 
\end{proposition}

In Proposition 2.2.2.7 of \cite{Lur09}, Lurie shows there is a Kan weak equivalence of simplicial sets
\begin{eqnarray*}
g: |S| _{Q^{\bullet}} \cong \underset{ \Delta^n \to S }{\text{colim}} \mathfrak{C}(J^n)(0,n+1) \to \underset{ \Delta^n \to S }{\text{colim}} \Delta^n \cong S
\end{eqnarray*}
for any simplicial set $S$. Hence, for a quasi-category $S$ and $x,y \in S_0$ we have a zig zag of Kan weak equivalences 
$$\text{Hom}^R_S(x,y) \xleftarrow{g} |\text{Hom}^R_S(x,y)|_{Q^{\bullet} } \xrightarrow{f} \mathfrak{C}(S)(x,y).$$
Now consider the above zig zag of Kan weak equivalences in the case $S= \text{Sing}(X)$ for a topological space $X$. There is a Kan weak equivalence of simplicial sets 
\begin{eqnarray*}
\theta: Hom^R_{\text{Sing}(X)}(x,y) \to \text{Sing}(P^M_{x,y}X)
\end{eqnarray*}
given as follows. A simplex  $\varphi: J^n \to \text{Sing}(X) \in Hom^R_{\text{Sing}(X)}(x,y)$ corresponds to a continuous map $\sigma_{\varphi}: |\Delta^{n+1}| \to X$ which collapses the last face of $|\Delta^{n+1}|$ to $x$ and sends the last vertex of $|\Delta^{n+1}|$ to $y$. For each point $p$ in the last face of $|\Delta^{n+1}|$ there is a straight line segment from $p$ to the last vertex of $|\Delta^{n+1}|$. These straight line segments give a family of disjoint paths inside $|\Delta^{n+1}|$ which start in the last face and end in the last vertex and such a family is parametrized by $|\Delta^n|$. The continuous map $\sigma_{\varphi}$ induces a continuous map $|\Delta^n| \to P^M_{x,y}X$ which corresponds to a simplex $\theta(\varphi): \Delta^n \to \text{Sing}(P^M_{x,y}X)$. The map $\theta$ is clearly a Kan weak equivalence of simplicial sets. It follows from the above zig zag formed by Kan weak equivalences $f$ and $g$  that $\mathfrak{C}(\text{Sing}(X))(x,y) \simeq \text{Sing}(P^M_{x,y}X)$. 
\\

\section{Algebraic models for loop spaces}

In this section we deduce an extension of a classical theorem of Adams from our previous results and discuss a few consequences. We start by showing that for a path connected pointed space $(X,b)$,  $\Lambda(\text{Sing}(X,b))(b,b)$  and $S_*(\Omega^M_bX;k)$ are weakly equivalent as dga's.

\begin{proposition}\label{7.8}
 Let $(X,b)$ be a pointed path connected topological space. The differential graded associative algebras $\Lambda(\text{Sing}(X,b))(b,b)$ and  $S_*(\Omega^M_bX;k)$ are weakly equivalent. 
\end{proposition}
\begin{proof} By definition $\Lambda(\text{Sing}(X,b))(b,b)= Q_{\square_c}(\mathfrak{C}_{\square_c}(\text{\text{Sing}}(X,b))(b,b))$. By Lemma \ref{7.2} we have a quasi-isomorphism of chain complexes
\begin{eqnarray*}
Q_{\square_c}(\mathfrak{C}_{\square_c}(\text{\text{Sing}}(X,b))(b,b)) \simeq Q_{\Delta}(|\mathfrak{C}_{\square_c}(\text{\text{Sing}}(X,b))(b,b)| ).
\end{eqnarray*} Moreover, this is quasi-isomorphism is a weak equivalence of dga's since the monoidal structures are preserved under the triangulation functor. By Proposition 5.3, we have an isomorphism
 \begin{eqnarray*}
 Q_{\Delta}(|\mathfrak{C}_{\square_c}(\text{\text{Sing}}(X,b))(b,b)|) \cong Q_{\Delta}(\mathfrak{C}(\text{\text{Sing}}(X,b))(b,b)).
 \end{eqnarray*}
  Finally, by Corollary \ref{7.7},  we have
  \begin{eqnarray*}
  Q_{\Delta}(\mathfrak{C}(\text{\text{Sing}}(X,b))(b,b)) \simeq S_*(\Omega_b^MX;k)
  \end{eqnarray*}
  as dga's.
  \end{proof}

In \cite{Ada52}, Adams introduced the cobar construction and constructed a chain map of dga's $\varphi: \Omega Q_{\Delta}( \text{Sing}(X,b) ) \to C^{\square}_*(\Omega^M_bX;k)$, where $C^{\square}_*(\Omega^M_bX;k)$ denotes the normalized singular cubical chains on $\Omega^M_bX$. Moreover, Adams showed that if $X$ is simply connected then $\varphi$ is a quasi-isomorphism. The proof of this fact relied on associating a spectral sequence to  $\Omega Q_{\Delta}( \text{Sing}(X,b) )$ and then comparing it to the Serre spectral sequence for the fibration $\Omega^M_bX \to PX \to X$. The simple connectivity assumption was used in order for the hypotheses of the Zeeman comparison theorem for spectral sequences to be satisfied. 
\\

We now deduce an extension of Adams' classical theorem (Corollary \ref{7.9} below) to the case when $X$ is a path connected space with possibly non-trivial fundamental group. Note that we have not relied on spectral sequence arguments but rather on  categorical and space level arguments as discussed in the previous section.
 
\begin{corollary} \label{7.9}
 For any pointed path connected space $(X,b)$, the differential graded algebras $\Omega(Q_{\Delta}(\text{Sing}(X,b)))$ and $S_*(\Omega^M_bX;k)$ are weakly equivalent.
\end{corollary}

\begin{proof}
This follows directly from Theorem \ref{7.1} and Proposition \ref{7.8}. 
\end{proof}

We conclude with two remarks and an application to model the free loop space.

 \begin{remark} \label{7.10}
 It follows from the above discussion that we may recover the homology of the based loop space of $|S|$ by taking the cobar construction on any connected dg coalgebra $\Omega$-quasi-isomorphic to $Q_{\Delta}^K(S)$. In general, $Q_{\Delta}(S)$ and $Q_{\Delta}^K(S)$ are quasi-isomorphic but not necessarily $\Omega$-quasi-isomorphic. However, if $S_0=\{x\}$ and $S_1=\{ s_0(x) \}$, where $s_0(x)$ denotes the degenerate $1$-simplex at $x$, then $Q_{\Delta}(S)$ and $Q_{\Delta}^K(S)$ are simply connected dg coalgebras and the natural map of dg coalgebras $\iota: Q_{\Delta}(S) \to Q_{\Delta}^K(S)$ is a quasi-isomorphism. Thus, by Propoisition 2.2.7 in \cite{LoVa12}, $\iota$ is an $\Omega$-quasi-isomorphism. Consequently, $\Omega Q_{\Delta}(S)$ is weakly equivalent as a dg algebra (i.e. quasi-isomorphic) to $S_*(\Omega^M_x|S|; k)$. 
  \end{remark}
 
 \begin{remark} In the case of a simplicial complex, an explicit and smaller model for the based loop space can be given using a Kan fibrant replacement functor. Let $K$ be a simplicial complex with an ordering of its vertices and let $v$ be a vertex of $K$. Let $fK$ be the simplicial set obtained by defining the face maps in accordance with the ordering of the vertices and adding degeneracies freely to $K$. The cobar construction on $Q_{\Delta}(fK)$ might not yield the homology of the based loop space of $|fK|$. However, we may consider the Kan fibrant replacement  $\text{Ex}^{\infty}(fK)$ of $fK$. $\text{Ex}^{\infty}(fK)$  is a Kan complex weakly equivalent to $fK$, so it follows that the Kan complexes $\text{Ex}^{\infty}(fK)$ and $\text{Sing}(|fK|)$ are weakly equivalent. Thus $\mathfrak{C}(\text{Ex}^{\infty}(fK))$, $\mathfrak{C}(\text{Sing}(|fK|))$, and $\text{Sing}(\mathcal{P}|fK|)$ are weakly equivalent simplicial categories. Therefore $\Lambda(\text{Ex}^{\infty}(fK))(v,v)$ is a dga model for the based loop space of $|fK|$ at $v$. This remark explains an example of Kontsevich outlined in \cite{Kon09}. In \cite{HT10}, a similar construction was also described for any simplicial set, which was then compared to Kan's loop group construction.
 \end{remark}
 Finally, a chain complex model for the free loop space of a path connected topological space may be obtained as follows. For any dga $A$ denote by $CH_*(A)$ the Hochschild chain complex of $A$. For the definition we refer the reader to any standard reference such as  \cite{Lod98}. 
 
 \begin{corollary} \label{7.12}
 For any pointed path connected space $(X,b)$, the Hochschild chain complex $CH_*(\Omega(Q_{\Delta}(\text{Sing}(X,b))))$ is quasi-isomorphic to $S_*(LX;k)$, the singular chains on the free loop space of $X$.
\end{corollary}
\begin{proof}
This is a direct consequence of the fact that the Hochschild chain complex of the dga $S_*(\Omega^MX;k)$ is quasi-isomorphic to $S_*(LX;k)$ (a theorem usually attributed to Goodwillie \cite{Goo85}), Corollary \ref{7.9}, and the invariance of Hochschild chains under weak equivalences of dga's.
\end{proof}
As explained in remark 2.23 of \cite{Hes16}, for any connected dg coalgebra $C$ there is a quasi-isomorphism of chain complexes 
$$coCH_*(C) \simeq CH_*(\Omega C)$$
where $coCH_*(C)$ denotes the coHochschild chain complex of $C$; we refer to \cite{Hes16} for definitions and further details. As a consequence, we obtain a model for the free loop space $LX$ of a path connected space space $X$ that does not require passing to the based loop space, which we expect to be convenient in studying string topology.
\begin{corollary} 
For any pointed path connected space $(X,b)$, the coHochschild complex $coCH_*( Q_{\Delta}(\text{Sing}(X,b)))$ is quasi-isomorphic to $S_*(LX;k)$. 
\end{corollary}
\begin{proof} This follows directly from Corollary \ref{7.2} and the fact that $coCH_*(C) \simeq CH_*(\Omega C)$ for any connected dg coalgebra $C$. 
\end{proof}.

\bibliographystyle{plain}

 \Addresses

\end{document}